\documentclass[a4paper,10pt]{article}
\usepackage[utf8]{inputenc} 
\usepackage[english]{babel} 

\usepackage{amsmath}   
\usepackage{amssymb}   
\usepackage{amsthm}    
\usepackage{amsfonts}  
\usepackage{mathrsfs}  
\usepackage{esint}
\usepackage{mathtools}
\usepackage{microtype}
\usepackage{authblk}

\usepackage[shortlabels, inline]{enumitem}

\usepackage{marginnote}

\usepackage[font=footnotesize,labelfont=bf]{caption}
\usepackage{subcaption}

\theoremstyle{plain}
\newtheorem{theorem}{Theorem}[section]
\newtheorem{corollary}[theorem]{Corollary}
\newtheorem{lemma}[theorem]{Lemma}
\newtheorem{proposition}[theorem]{Proposition}
\newtheorem{definition}[theorem]{Definition}

\theoremstyle{remark}
\newtheorem{remark}[theorem]{Remark}

\usepackage{xcolor}
\usepackage[pdfpagelabels]{hyperref}
\definecolor{green}{HTML}{2ECC71}
\definecolor{blue}{HTML}{3498DB}
\definecolor{red}{HTML}{E74C3C}
\hypersetup{%
  colorlinks,
  linktocpage,
  urlcolor  = blue,
  citecolor = green,
  linkcolor = red}

\DeclarePairedDelimiterX{\hsp}[2]{\langle}{\rangle}{#1, #2}
\DeclarePairedDelimiterX{\pairing}[2]{\langle}{\rangle}{#1 \mid #2}
\DeclarePairedDelimiter{\bracket}{\lbrack}{\rbrack}

\newcommand{\abs}[1]{\left\lvert#1\right\rvert}						

\newcommand{\norm}[1]{\left\lVert#1\right\rVert}					

\newcommand{\set}[1]{\left\{#1\right\}}							

\newcommand{\tparen}[1]{\big({#1}\big)}							
\newcommand{\paren}[1]{\left(#1\right)}							

\newcommand{\comma}{\,\,\mathrm{,}\;\,}

\newcommand{\fstop}{\,\,\mathrm{.}}

\newcommand{\emparg}{{\,\cdot\,}}								

\def\N{{\mathbb N}}    
\def\Z{{\mathbb Z}} 
\def\R{{\mathbb R}}

\newcommand{\vol}{\mathsf{vol}}

\def\Leb{\mathcal{L}}

\begin{document}

\title{
Polyharmonic Fields and
Liouville Quantum Gravity Measures on Tori of Arbitrary Dimension: from Discrete to Continuous\\[3cm]}

\author[1]{Lorenzo Dello Schiavo} 
\author[2]{Ronan Herry}
\author[3]{Eva Kopfer}
\author[3]{Karl-Theodor Sturm}
\affil[1]{Institute of Science and Technology Austria \authorcr lorenzo.delloschiavo{@}ist.ac.at \vspace{.2cm}}
\affil[2]{IRMAR, Université de Rennes 1
  \par
  263 avenue du Général Leclerc, 35042 Rennes Cedex
  \par
 ronan.herry@univ-rennes1 \vspace{.2cm}
  }
\affil[3]{%
Institute for Applied Mathematics --
University of Bonn
\authorcr
eva.kopfer@iam.uni-bonn.de
\authorcr
sturm@uni-bonn.de
}

\maketitle

\begin{abstract}\noindent
For an arbitrary dimension $n$, we study:
\begin{itemize}
  \item the Polyharmonic Gaussian Field $h_L$ on the discrete torus $\mathbb{T}^n_L = \frac{1}{L} \mathbb{Z}^{n} / \mathbb{Z}^{n}$, that is the random field whose law on $\mathbb{R}^{\mathbb{T}^{n}_{L}}$  given by
\begin{equation*}
  c_n\, e^{-b_n\norm{(-\Delta_L)^{n/4}h}^2} dh,
\end{equation*}
where $dh$ is the Lebesgue measure and $\Delta_{L}$ is the discrete Laplacian;
\item the associated discrete Liouville Quantum Gravity measure associated with it, that is the random measure on $\mathbb{T}^{n}_{L}$  
  \begin{equation*}
    \mu_{L}(dz) = \exp \paren{ \gamma h_L(z) - \frac{\gamma^{2}}{2} \mathbf{E} h_{L}(z) } dz,
  \end{equation*}
  where $\gamma$ is a regularity parameter.
\end{itemize}
As $L\to\infty$, we prove convergence of the fields $h_L$ to the Polyharmonic Gaussian Field $h$ on the continuous torus $\mathbb{T}^n = \mathbb{R}^{n} / \mathbb{Z}^{n}$, as well as convergence of the random measures $\mu_L$ to the LQG measure $\mu$ on $\mathbb{T}^n$, for all $\abs{\gamma} < \sqrt{2n}$.
\end{abstract}

\tableofcontents
%
%
%
%
\section*{Introduction}

We study Gaussian random fields and the associated LQG measures on continuous and discrete tori of arbitrary dimension. The random field $h$ on the continuous torus is a particular case of the co-polyharmonic field introduced and analyzed in detail in \cite{DHKS} in great generality on all `admissible' manifolds of even dimension. One of the main goals now is to study the approximation of these fields and the associated LQG measures by their discrete counterparts.

The \emph{polyharmonic fields} $h$ on $\mathbb{T}^n\cong [0,1)^n$ and $h_L$ on $\mathbb{T}^n_L\cong \{0,\frac1L,\ldots, \frac{L-1}L\}^n$ for $L\in\N$ are centered Gaussian random fields with covariance functions
\begin{align*}\mathbf E\Big[h(x)\,h(y)\Big]&=k(x,y)\coloneqq  \frac1{a_n}\, \mathring G^{n/2}(x,y)\comma\\
\mathbf E\Big[h_L(x)\,h_L(y)\Big]&=k_L(x,y)\coloneqq \frac1{a_n}\, \mathring G_L^{n/2}(x,y)\fstop
\end{align*}
 given in terms of the integral kernel for the `grounded' inverse of the (continuous and discrete, resp.) poly-Laplacian $(-\Delta)^{n/2}$ and $(-\Delta_L)^{n/2}$, and a normalization constant $a_n\coloneqq \frac2 {\Gamma(n/2)\, (4\pi)^{n/2}}$. With this choice of the normalization constant 
\begin{lemma}[cf.~Lemma \ref{log-div}]
\[
\Big|k(x,y)-\log\frac1{d(x,y)}\Big|\le C \fstop
\]
\end{lemma}

\paragraph{Characterization of the Discrete Polyharmonic Field.}
Let $n,L\in\N$ be given and assume for convenience that $L$ is odd, let $\Z^n_L=\{-\frac{L-1}2,-\frac{L-1}2+1,\ldots, \frac{L-1}2\}^n$, and set $N\coloneqq L^n$ and
$$c_n\coloneqq \left(\frac{a_n}{2\pi N}\right)^{\frac{N-1}2}\cdot
\prod_{z\in\Z^n_L\setminus\{0\}}\left(
4L^2\,\sum_{k=1}^n \sin^2\big(\pi z_k/L\big)\right)^{n/4}
\fstop$$
Define the  measure $\widehat{\boldsymbol\nu}(h)$ on $\R^{N}\cong\R^{\mathbb{T}_L^n}$
by
$$d\widehat{\boldsymbol\nu}(h)\coloneqq c_n\,e^{-\frac{a_n}{2N}\|(-\Delta_L)^{n/4}h\|^2}\,d\Leb^N(h)\comma$$
 and denote by $\boldsymbol\nu$ its push forwards under the map 
\[
h\mapsto \mathring h, \quad \mathring h_v\coloneqq h_v-\frac1{N} \sum_{v=1}^{N} h_v \fstop
\]
 In other words,
 $\boldsymbol\nu=\widehat{\boldsymbol\nu}\left(\emparg  \middle | \sum_{v=1}^{N} h_v=0\right).$
 
Furthermore,
$$\mathring T_*\boldsymbol\nu=\mathring{\mathbf P}\qquad \mathring T^{-1}_*\mathring{\mathbf P}= \boldsymbol\nu$$
where $\mathring{\mathbf P}$ denotes the distribution of the `grounded white noise'  on ${\mathbb{T}_L^n}$, explicitly given as
$$d\mathring{\mathbf P}(\Xi)=
\frac1{(2\pi)^{\frac{N-1}2}}\,e^{-\frac{1}{2N}\| \Xi\|^2}\,d\Leb^{N-1}_H(\Xi)$$
on the hyperplane $H=\{\Xi\in \R^N: \ \sum_{v=1}^{N} \Xi_v=0\}$,
and where 
\[
\mathring T:  h\mapsto \Xi=\sqrt{a_n}\, (-\Delta_L)^{n/4}h\qquad
\mathring T^{-1}:  \Xi\mapsto h=\frac1{\sqrt{a_n}}\, \mathring G_L^{n/4}\Xi \fstop
\]

 \begin{theorem}[cf.~Thm.~\ref{repres}] 
 The distribution of the discrete polyharmonic field on ${\mathbb{T}_L^n}$ is given by the probability measure $\boldsymbol\nu$  on $\R^{\mathbb{T}_L^n}\cong\R^N$. 
\end{theorem}

\paragraph{Convergence of the Random Fields.} As $L\to\infty$,
the polyharmonic fields $h_L$ on the discrete tori converge to the polyharmonic field $h$ on the continuous torus.
This convergence of the fields, indeed, holds in great generality. 

For a precise formulation, one either has to specify  classes of test functions on $\mathbb{T}^n$ which admit traces on $\mathbb{T}^n_L$, or  unique ways of extending functions on $\mathbb{T}^n_L$ onto $\mathbb{T}^n$.
\begin{theorem}[cf.~Thm.~\ref{thm-hL}] For all $f\in \bigcup_{s>n/2}\mathring H^{s}(\mathbb{T}^n)$,
$$\langle  h_L,f\rangle_{\mathbb{T}^n_L} \to \langle h,f\rangle_{\mathbb{T}^n} \quad \text{in $L^2(\mathbf P)$ as }L\to\infty.$$
\end{theorem}

Let $\mathcal D_L\subset\mathcal C^\infty(\mathbb{T}^n)$ denote the linear span of the eigenfunctions $\varphi_z$ for the negative Laplacian with associated eigenvalues $0<\lambda_z<(L\pi)^2$, or more explicitly,
$$\mathcal D_L\coloneqq\left\{f: \ f(x)=\sum_{z\in \Z^n_L}\Big[ \alpha_z\cos(2\pi x\cdot z)+\beta_z\sin(2\pi x\cdot z)\Big],\  \alpha_z,\beta_z\in\R
\right\}\fstop$$
\begin{theorem}[cf. Thm.~\ref{thm3}] For all $f\in \mathring H^{-n/2}(\mathbb{T}^n)$,
$$\langle  h_{L,\sharp},f\rangle_{\mathbb{T}^n} \to \langle h,f\rangle_{\mathbb{T}^n} \quad \text{in $L^2(\mathbf P)$ as }L\to\infty$$
where $h_{L,\sharp}^\omega$ for every $\omega$ denotes the unique function in $\mathcal D_L$ which coincides with $h_L^\omega$ on $\mathbb{T}^n_L$.
\end{theorem}

The same convergence assertion also holds for the so-called \emph{spectrally reduced polyharmonic field} $h_L^{{-\circ}}$ on $\mathbb{T}^n_L$ 
given in terms of the eigenbasis $\{\varphi_z\}_{z\in\Z^n_L}$ of the discrete Laplacian $\Delta_L$ as
\[
h^{{-\circ}}_L(x)\coloneqq \sum_{z\in\Z^n_L}
\left(\frac{L^2}{\pi^2|z|^2}\sum_{k=1}^n \sin^2(\pi z_k/L)\right)^{-n/4}\, \langle h_L| \varphi_z\rangle_{\mathbb{T}^n_L}\cdot \varphi_z(x)\fstop
\]


Our convergence results apply to the case of arbitrary dimension~$n$. In dimension~$n\leq 4$, several results are available in the literature for the convergence of other discrete Fractional Gaussian Fields of integer order to the corresponding counterpart in the continuum, including e.g., the \emph{odometer} for the \emph{sandpile model}, or the \emph{membrane model}.
For a comparison of these results with those in the present work, see~\S\ref{remarks} below.

\paragraph{Convergence of the Random Measures.} The convergence questions for the associated random measures are more subtle. Again, of course, one expects that
the Liouville measure~$\mu_L$ on the discrete tori converge as~$L\to\infty$ to the Liouville measure $\mu$ on the continuous torus.
This convergence of the random measure, however, only holds for small parameters $\gamma$.

\begin{theorem}[cf.~Thm.~\ref{conv-disc-meas}] 
Assume $|\gamma|<\sqrt\frac{n}{e}$, and let $a$ be an odd integer $\ge2$. Then in $\mathbf P$-probability and in $L^1(\mathbf P)$,   
$$\mu_{a^\ell}\to\mu\qquad\text{as }\ell\to\infty.$$ 
\end{theorem}

Analogous convergence results hold for the random measures associated with the Fourier extensions of the discrete polyharmonic fields and the reduced discrete polyharmonic field, in the latter case even in the whole range of subcriticality $\gamma\in (-\sqrt{2n},\sqrt{2n})$. 

\begin{theorem}[cf.~Thm.~\ref{conv-four-meas}, Thm.~\ref{conv-mod-meas}] For $|\gamma|<\sqrt{\frac{n}{e}}$,
$$\mu_{L,\sharp}\to \mu\qquad\text{as }L\to\infty,$$
and for $|\gamma|< \sqrt{2n}$,
$$\mu^{{-\circ}}_{L,\sharp}\to \mu\qquad\text{as }L\to\infty.$$

\end{theorem}

\paragraph{Uniform Integrability of the Random Measures.}
As an auxiliary result of independent interest, we provide a direct proof of the uniform integrability of (discrete, semi-discrete, and continuous) random measures on the multidimensional torus.

\begin{theorem}[cf.~Thm.~\ref{p:Estimate}]\label{uni-int-crit} Assume that $|\gamma|< 
\sqrt\frac{n}{e}$. 
Then
$$\sup_L\mathbf E\Big[\big|\mu_L(\mathbb{T}^n_L)\big|^2\Big]<\infty$$
and
$$\sup_L\mathbf E\Big[\big|\mu_{L,\sharp}(\mathbb{T}^n)\big|^2\Big]<\infty\fstop$$

\end{theorem}
\paragraph{\bf Acknowledgements}
KTS is  grateful to Christoph Thiele for valuable discussions and
helpful references.
LDS is grateful to Nathana\"el Berestycki for valuable discussions on Gaussian Multiplicative Chaoses.

LDS gratefully acknowledges financial support from the European Research Council 
(grant agreement No 716117, awarded to J.~Maas) and from the Austrian Science Fund (FWF) through project F65.
He also acknowledges funding of his current position from the Austrian Science Fund (FWF) through project ESPRIT~208.
RH, EK, and KTS gratefully acknowledge funding by the Deutsche Forschungsgemeinschaft through the project `Random Riemannian Geometry' within the SPP 2265 `Random Geometric Systems', 
 through the Hausdorff Center for Mathematics (project ID 390685813), and through  project B03  within the CRC 1060 (project ID 211504053).
 RH and KTS also  gratefully acknowledges financial support from the European Research Council through the ERC AdG `RicciBounds' (grant agreement 694405).

Data sharing not applicable to this article as no datasets were generated or analyzed during the current study.

\section{Laplacian and Kernels on Continuous and Discrete Tori}

\subsection{Laplacian and Kernels on the Continuous Torus}

\paragraph{(a)}
For $n\in\N$, we denote by $\mathbb{T}^n\coloneqq \left(\R/\Z\right)^n$ the continuous $n$-torus. 
Where it seems helpful, one can always think of the torus  $\mathbb{T}^n$ as the set $[0,1)^n\subset\R^n$. 
It inherits from $\R^n$ the additive group structure and the Lebesgue measure, denoted in the sequel by $d\Leb^n(x)$ or simply by $dx$. The distance on $\mathbb{T}^n$ is given by
\[
d(x,y)\coloneqq \left(\sum_{k=1}^n \big( |x_k-y_k|\wedge|1-x_k+y_k| \big)^2\right)^{1/2} \fstop
\]

\paragraph{(b)}
For $z\in\Z^n$ and $x\in\mathbb{T}^n$ put
\[
\Phi_z(x)\coloneqq \exp\Big(2\pi i\, z\cdot x\Big) \fstop
\]
The family $(\Phi_z)_{z\in \Z^n}$ is a complete ON basis of $L^2_{\mathbb C}(\mathbb{T}^n)$. It consists of eigenfunctions of the negative Laplacian $-\Delta=-\sum_{k=1}^n\frac{\partial^2}{\partial x^2_k} $ on $\mathbb{T}^n$ with corresponding eigenvalues given by
 \[
 \lambda_z\coloneqq (2\pi |z|)^2\fstop
 \]

\paragraph{(c)}
The Fourier transform of the function $f\in L^2_{\mathbb C}(\mathbb{T}^n)$ is the function (or ``sequence")  $g\in \ell^2(\Z^n)$ given by
\[
g(z)\coloneqq \langle f,\Phi_z\rangle_{\mathbb{T}^n}\coloneqq \int_{\mathbb{T}^n} f(x)\, \overline\Phi_z(x)\,dx\fstop
\]
Conversely, for $g$ as above and a.e.~$x\in\mathbb{T}^n$,
\[
f(x)=\sum_{z\in\Z^n} g(z) \Phi_z(x) \fstop
\]

\paragraph{(d)}
To obtain a complete ON basis $(\varphi_z)_{z\in\Z^n}$ for the real $L^2$-space, choose a subset $\hat\Z^n$ of $\Z^n\setminus\{0\}$ with
\[
\Z^n\setminus\{0\}=\hat\Z^n\ \dot\cup\  \big(-\hat\Z^n\big)\comma
\]
and define
\[
\varphi_z(x)\coloneqq \begin{cases}
\frac1{\sqrt2}\big(\Phi_z+\Phi_{-z}\big)(x)=\sqrt 2\, \cos\big(2\pi\,z\cdot x\big)\quad&\text{if }z\in\hat\Z^n\comma
\\
\frac1{\sqrt2\, i}\big(\Phi_z-\Phi_{-z}\big)(x)=\sqrt 2\, \sin\big(2\pi\,z\cdot x\big)\quad&\text{if }z\in-\hat\Z^n\comma
\\
{\bf 1}&\text{if }z=0 \fstop
\end{cases}
\]

\paragraph{(e)}
Functions $f$ on $\mathbb{T}^n$ will be called \emph{grounded} if $\int f\,d\Leb^n=0$.
For  $s\in\R$, the (grounded, real) Sobolev space $\mathring H^s(\mathbb{T}_n)$ can be identified with  a set of formal series:
\[
\mathring H^s(\mathbb{T}^n)=\left\{f=\sum_{z\in\Z^n\setminus\{0\}}\alpha_z\varphi_z: \ \alpha_z\in\R, \ \sum_{z\in\Z^n\setminus\{0\}}|z|^{2s}\,|\alpha_z|^2<\infty\right\}
\fstop
\]
Then for all $f=\sum_{z\in\Z^n\setminus\{0\}}\alpha_z\varphi_z\in \mathring H^r(\mathbb{T}^n)$ and $g=\sum_{z\in\Z^n\setminus\{0\}}\beta_z\varphi_z\in \mathring H^s(\mathbb{T}^n)$ with $r+s\ge0$,
$$\langle f,g\rangle_{\mathbb{T}^n}=\sum_{z\in\Z^n\setminus\{0\}}\alpha_z\,\beta_z.$$
The norm of~$\mathring{H}^s(\mathbb{T}^n)$ is given by the square root of 
$\sum_{z\in\Z^n\setminus\{0\}}|z|^{2s}\,|\alpha_z|^2$. Equivalently it could be defined 
 with~$\lambda_z^s$ in place of~$|z|^{2s}$. This is the convention adopted in~\cite{DHKS}.
The two norms differ only by a factor~$(2\pi)^{s}$.



\paragraph{(f)} Given any function $u: \Z^n\to{\mathbb C}$ we define the \emph{principal value along cubes} of the series $\sum_z u(z)$ by
$$\sum^{\Box}_{z\in\Z^n} u(z):=\lim_{L\to\infty}\sum_{z\in\Z^n, \ \norm{z}_{\infty} <L/2}u(z)$$
provided the latter limit exists in $\mathbb C$ or in $\R\cup\{\pm\infty\}$.

\paragraph{(g)}
Since $M$ is compact, there exists a unique \emph{grounded Green kernel} $\mathring{G}$ satysfying
\begin{equation*}
\mathring{G}(x,y) \simeq \abs{x-y}^{2-n}.
\end{equation*}
In particular, $\mathring{G} \in L^{p}(M \times M)$ for all $p < \frac{n}{n-2}$.
We claim that we have:
\begin{align*}
  \mathring G(x,y)& = \sum^\Box_{z\in\Z^n\setminus\{0\}}\frac1{(2\pi |z|)^2}\,\varphi_z(x)\,  \varphi_z(y)= \sum^\Box_{z\in\Z^n\setminus\{0\}}\frac1{(2\pi |z|)^2}\,\Phi_z(x)\, \overline\Phi_z(y)\\
&
=
\sum^\Box_{z\in\Z^n\setminus\{0\}}\frac1{(2\pi |z|)^2}\,\Phi_z(x-y)
=\sum^\Box_{z\in\Z^n\setminus\{0\}}\frac1{(2\pi |z|)^2}\,\cos\Big(2\pi \,z\cdot(x-y)\Big),
\end{align*}
where the convergence holds almost everywhere and in $L^{p}$ for $p < n/(n-2)$.
Indeed consider the filtration $(\mathfrak{F}_{L})$ where $\mathfrak{F}_{L}$ is the $\sigma$-algebra generated by the $\varphi_{z}$ for $z \in \mathbb{Z}^{n}$, $\norm{z}_{\infty} < L/2$, and the associated closed martingale $\mathring G_{L} = \mathbf{E} \bracket{ \mathring{G} | \mathfrak{F}_{L}}$, where expectation is with respect to $\vol \otimes \vol$.
Take $z \in \mathbb{Z}^{n}$ with $\norm{z}_{\infty} < L/2$.
Since $\varphi_{z}$ is $\mathfrak{F}_{L}$-measurable, we get that
\begin{equation*}
  \int \mathring{G}_{k}(x,y) \varphi_{z}(y) dy = \int \mathring{G}(x,y) \varphi_{z}(y) dy = \paren{-\Delta}^{-1} \varphi_{z}(x) = \lambda_{z}^{-1} \varphi_{z}(x).
\end{equation*}
On the other hand, when $\norm{z}_{\infty} \ge L/2$, since the $\varphi_{z}$'s form an orthonormal basis, we find that $\mathbf{E}\bracket{ \varphi_{z} | \mathfrak{F}_{L}}= 0$, and thus
\begin{equation*}
  \int \mathring{G}_{L}(x,y) \varphi_{z}(y) dy = 0.
\end{equation*}
This shows that
\begin{equation*}
  \mathring{G}_{L}(x,y) = \sum_{z \in \mathbb{Z}^{n}_{L} \setminus \{0\}} \lambda_{z}^{-1} \varphi_{z}(x) \varphi_{z}(y),
\end{equation*}
and thus  the almost everywhere convergence of the series follows by the martingale convergence theorem.


\paragraph{(h)}
The \emph{polyharmonic operator} is defined as
\begin{equation*}
  a_n\cdot(-\Delta)^{n/2}
\qquad \text{with} \qquad a_n\coloneqq \frac2 {\Gamma(n/2)\, (4\pi)^{n/2}}\fstop
\end{equation*}
The inverse operator admits a kernel denoted by $k$.

As for the Green kernel, we have the following representation.
\begin{lemma}\label{green-kernel-series}
  We have that
  \begin{equation*}
    k = \sum^{\Box}_{z\in \mathbb{Z}^n\setminus\{0\}}\frac1{{(2\pi|z|)^n}}\,\varphi_z\otimes\varphi_z,
  \end{equation*}
  where the series converges in $L^2(\mathbb{T}^n\times\mathbb{T}^n)$ and almost-everywhere.
\end{lemma}

\begin{remark} We conjecture 
that the convergence indeed holds everywhere but do not have a proof of this fact.
\end{remark}

\begin{proof} 
  Since the series on the right-hand side is orthogonal, we find that
  \begin{equation*}
    \norm{\sum_{z \in \mathbb{Z}^{n} \setminus \{0\}} \lambda_{z}^{-n/2} \varphi_{z} \otimes \varphi_{z}}_{L^{2}} = \sum_{z \in \mathbb{Z}^{n} \setminus \{0\} } \paren{2 \pi z}^{-2n} < \infty.
  \end{equation*}
  This shows that the series actually converges in $L^{2}$.
  The rest of the claim is obtained by a martingale argument as for the previous lemma.
\end{proof}

\begin{lemma}
The function 
\begin{equation}\label{eq:l:HKRepresentation:0}
f\colon x\longmapsto k(x,0)=\frac1{\Gamma(n/2)}\int_0^\infty \mathring p_t(x,0)\, t^{n/2-1}dt
\end{equation}
is differentiable at every~$x\in \mathbb{T}^n\setminus\set{0}$, and, for every~$k\leq n$,
\begin{equation}\label{eq:l:HKRepresentation:0.1}
\frac{\partial}{\partial x_k}f(x)=\frac1{\Gamma(n/2)}\int_0^\infty  \frac{\partial}{\partial x_k}\mathring p_t(x,0)\, t^{n/2-1}dt \fstop
\end{equation}
\end{lemma}
\begin{proof}
The heat-kernel representation in~\eqref{eq:l:HKRepresentation:0} holds as in~\cite[Lem.~2.4]{DKS20}.
For fixed~$k\leq n$, standard Gaussian upper heat kernel estimates provide the summability of the right-hand side in~\eqref{eq:l:HKRepresentation:0.1}, hence~\eqref{eq:l:HKRepresentation:0.1} follows by differentiation under integral sign.
Since~$x\mapsto \tparen{\frac{\partial}{\partial x_k}\mathring p_t}(x,0)$ is continuous for every~$k$ on the whole of~$\mathbb{T}^n$, we have that~$\frac{\partial}{\partial x_k} f$ is continuous away from~$0$, and the differentiability of~$f$ follows by standard arguments in multivariate calculus.
\end{proof}

The 
constant $a_n$ is chosen such that it leads to a precise logarithmic divergence of $k$.
\begin{lemma}[\cite{DHKS}]\label{log-div} $\exists C=C(n): \forall x,y\in\mathbb{T}^n$:
\begin{equation}
\left|k(x,y)-\log \frac1{d(x,y)}\right|\le C\fstop
\end{equation}
\end{lemma}

\begin{proof} Note  that the estimate in Proposition 2.13 in \cite{DHKS} for the kernel $G^{n/2}(x,y)$ of the $n/2$-power of the Green operator not only holds for even but  also for odd $n$.
\end{proof}


\subsection{Laplacian and Kernels on the Discrete Torus}

\paragraph{(a)}
 For the sequel, fix  $L\in\N$. For convenience, we assume that $L$ is odd. Put 
\[
\Z^n_L\coloneqq \{z\in\Z^n: \, \|z\|_\infty\coloneqq\max_{k=1,\ldots,n}|z_k|<L/2\}\comma
\]
 and let
\[
\mathbb{T}^n_L\coloneqq \big(\tfrac1L \Z\big)^n\big/\Z^n
\]
denote the discrete $n$-torus with edge length $\frac1L$. 
Where  helpful, one can  think of the discrete torus  $\mathbb{T}_L^n$ as the set 
$\frac1L\Z^n_L=\{\frac{k}L: \ 
k\in\Z,\, 0\le k<L
\})^n\subset\R^n$. We always regard it as a subset of the continuous torus $\mathbb{T}^n$.
Furthermore, let
\[
m_L\coloneqq \frac1{L^n}\sum_{z\in\mathbb{T}_L^n}\delta_z
\]
denote the normalized counting measure on $\mathbb{T}^n_L$.
Points $v,u\in \mathbb{T}^n_L$ are \emph{neighbors}, in short $v\sim u$, if $
d(v,u)=\frac1L$. Each point in $\mathbb{T}^n_L$ has $2n$ neighbors.

\paragraph{(b)}
We define the \emph{discrete Laplacian} $\Delta_L$ acting on functions $f\in L^2(\mathbb{T}^n_L)$ by
\[
\Delta_L f(v)\coloneqq  L^2\cdot\sum_{u\sim v}\Big[f(u)-f(v)\Big]=2nL^2 \big({\sf p}_Lf-f\big)(v)
\]
with the transition kernel on $\mathbb{T}^n_L$ given by
\[
p_L(v,u)\coloneqq \frac{L^n}{2n}{\bf 1}_{v\sim u}
\]
and its action by
 $({\sf p}_Lf)(v)=L^{-n} \sum_u p_L(v,u) f(u)$.
Furthermore,
we define the grounded transition kernel by $\mathring p_L(v,u)\coloneqq p_L(v,u)-1$.

The \emph{discrete Green operator} acting on grounded functions $f\in \mathring L^2(\mathbb{T}^n_L)$ is defined by
\begin{equation*}
\mathring{\sf G}_L f\coloneqq \frac1{2nL^2}\, \sum_{k=0}^\infty {\sf p}_L^kf=\frac1{2nL^2}\, \sum_{k=0}^\infty \mathring {\sf p}_L^kf \fstop
\end{equation*}
In particular, the grounded discrete Green kernel is given by
\[
\mathring G_L(v,u)=\frac1{2nL^2}\, \sum_{k=0}^\infty \mathring p_L^k(v,u)
\]
and its action by
$(\mathring{\mathsf{G}}_Lf)(v)=L^{-n} \sum_y \mathring{G}_L(v,u) f(u)$.

\paragraph{(c)}
A complete ON basis of the complex $L^2_{\mathbb C}(\mathbb{T}^n_L,m_L)$ is given by $(\Phi_z)_{z\in \Z^n_L}$ with
\[
\Phi_z(v)\coloneqq \exp\Big(2\pi i\, z\cdot v\Big) \qquad (\forall v\in\mathbb{T}^n_L)\fstop
\]
The functions $\Phi_z$ are (normalized) eigenfunctions of the negative discrete Laplacian $-\Delta_L$ with eigenvalues
\begin{equation}\label{lL}
\lambda_{L,z}\coloneqq 4L^2\,\sum_{k=1}^n \sin^2\Big(\pi z_k/L\Big)\fstop
\end{equation}
Note that as $L\to\infty$, the RHS converges to $\lambda_z=(2\pi |z|)^2$ for any $z\in\Z^n$.

A  complete ON basis of $L^2_{\mathbb R}(\mathbb{T}^n_L,m_L)$ is given by the functions $\varphi_z$ for $z\in\Z^n_L$ where as before   $\varphi_0\equiv1$ and $\varphi_z(v)=\sqrt 2\, \cos\big(2\pi\,z\cdot v\big)$ if $z\in\hat\Z^n\cap\Z^n_L$ and $\varphi_z(v)=\sqrt 2\, \sin\big(2\pi\,z\cdot v\big)$ if $z\in(-\hat\Z^n)\cap\Z^n_L$.

\begin{remark} For even $L$, the previous definitions require some modifications. The set $\Z^n_L$ has to be re-defined as
$$\Z^n_L\coloneqq  \big\{-{L}/2+1,\ldots, {L}/2-1,{L}/2\big\}^n\fstop$$
Each $z\in\Z^n_L$ we decompose into $z'\coloneqq (z_k)_{k\in\sigma_z}$ and $\tilde z\coloneqq (z_k)_{k\in\tau_z}$ with
$\sigma_z\coloneqq\{k\in \{1,\ldots, n\}: z_k=L/2\}$, $\tau_z\coloneqq\{k\in \{1,\ldots, n\}: z_k<L/2\}$.
Similarly, for $v\in\mathbb{T}^n_L$ we put $v'\coloneqq (v_k)_{k\in\sigma_z}$ and $\tilde v\coloneqq (v_k)_{k\in\tau_z}$. Then
$$\Phi_z(v)=(-1)^{L\,|v'|_{\sigma_z}}\cdot \Phi_{\tilde z}(\tilde v)\quad\text{with }|v'|_{\sigma_z}\coloneqq \sum_{k\in\sigma_z}v_k\fstop$$
Thus a complete ON basis of $L^2_{\mathbb R}(\mathbb{T}^n_L,m_L)$ is given by the functions 
\begin{equation}\varphi_z(v)\coloneqq(-1)^{L\,|v'|_{\sigma_z}}\cdot \varphi_{\tilde z}(\tilde v)\comma \quad z\in Z^n_L\comma\end{equation}
where $\varphi_{\tilde z}$ for $\tilde z\in\Z^n$ with $\|\tilde z\|_\infty<L/2$ is defined as before.
\end{remark}
\paragraph{(d)}
In terms of the discrete eigenfunctions, the \emph{discrete grounded Green kernel}, the integral kernel of the inverse of $-\Delta_L$ acting on grounded $L^2$-functions, is given  as
\begin{align*}\mathring G_L(v,u)&= \sum_{z\in\Z^n_L\setminus\{0\}} \frac1{\lambda_{L,z}}\varphi_{z}(v)\, \varphi_z(u)= \sum_{z\in\Z^n_L\setminus\{0\}} \frac1{\lambda_{L,z}}\Phi_z(v)\,\overline \Phi_z(u)\\
&
=\sum_{z\in\Z^n_L\setminus\{0\}} \frac1{4L^2\,\sum_{k=1}^n \sin^2\Big(\pi z_k/L\Big)}\cdot 
\cos\Big(2\pi \, z\cdot (v-u)\Big)\comma
\end{align*}
and the \emph{discrete polyharmonic kernel}, the integral kernel of the inverse of $a_n(-\Delta_L)^{n/2}$ acting on grounded $L^2$-functions, as
\begin{align}k_L(v,u)&=\frac1{a_n}\sum_{z\in\Z^n_L\setminus\{0\}} \frac1{\lambda^{n/2}_{L,z}}\varphi_z(v)\, \varphi_z(u)=\frac1{a_n}\sum_{z\in\Z^n_L\setminus\{0\}} \frac1{\lambda_{L,z}^{n/2}}\Phi_z(v)\,\overline \Phi_z(u)
\nonumber
\\
\label{kL}
&=\frac1{a_n}\sum_{z\in\Z^n_L\setminus\{0\}} \frac1{\Big(4L^2\,\sum_{k=1}^n \sin^2\big(\pi z_k/L\big)\Big)^{n/2}}\cdot 
\cos\Big(2\pi \, z\cdot (v-u)\Big)\fstop
\end{align}

\begin{figure}
\includegraphics[scale=.45]{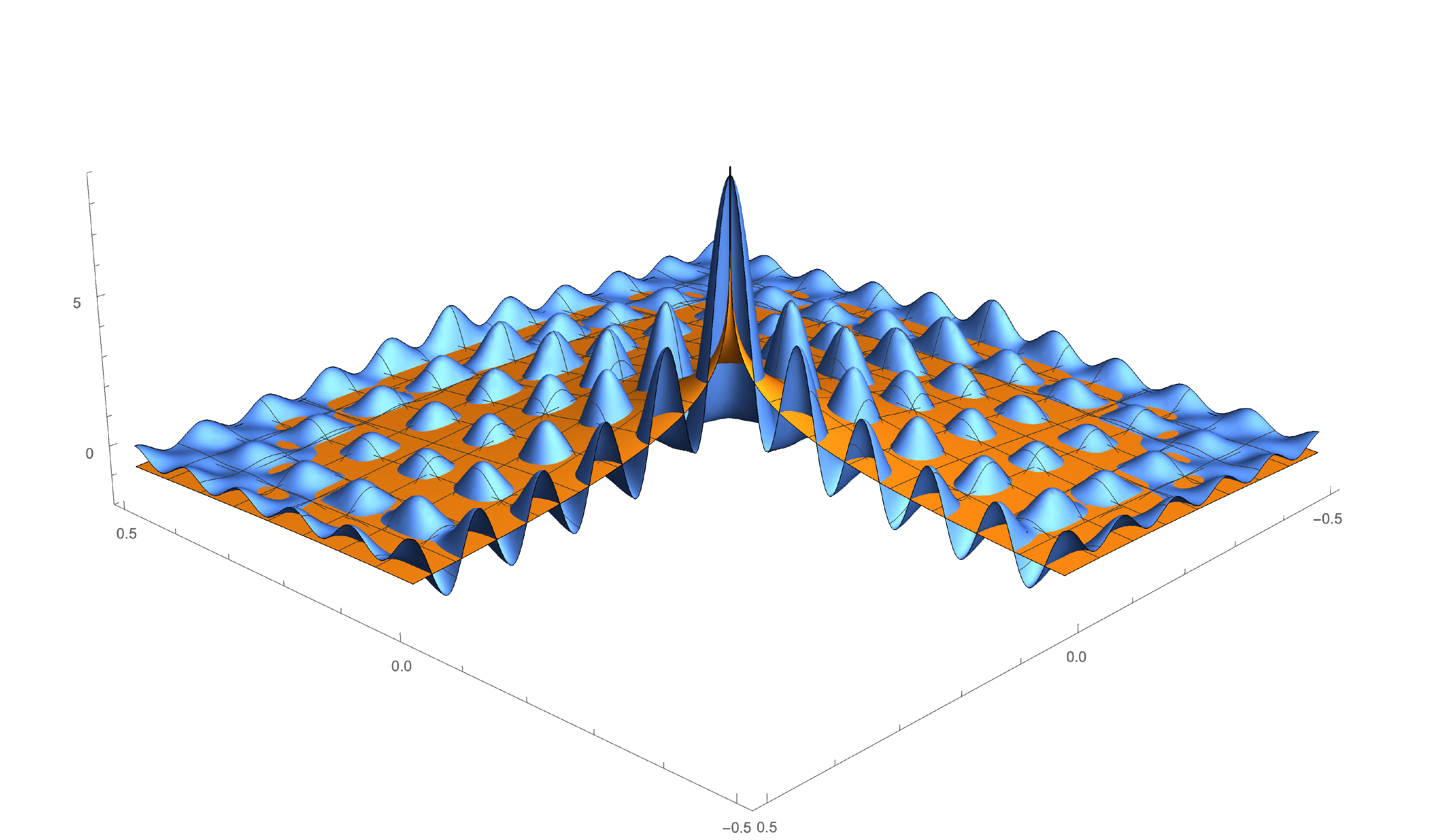}
\caption{$k(0,y)$ (orange) and $k_{11}(0,y)$ (blue) for~$y\in \mathbb{T}^2$ (sectional view with one quadrant removed).
}
\end{figure}

\subsection{Extensions and Projections}
\paragraph{(a) Piecewise Constant Extension/Projection.} \label{cons-proj-ext}
 Set $Q_L\coloneqq[-\frac1{2L},\frac1{2L})^n$ and $Q_L(v)\coloneqq v+Q_L$ for $v\in\mathbb{T}^n_L$. Observe that
$$\mathbb{T}^n=\dot{\bigcup\limits_{v\in\mathbb{T}^n_L}} Q_L(v)\fstop$$
Functions on $\mathbb{T}^n$ are called \emph{piecewise constant} if they are constant on each of the cubes $Q_L(v), v\in\mathbb{T}^n_L$.
Every function $f$ on the discrete torus $\mathbb{T}^n_L$ can be uniquely \emph{extended} to a piecewise constant function $f_{L,\flat}(x)$ by setting $f_{L,\flat}(x)\coloneqq f(v)$ if $x\in Q_L(v)$. In other words,
$$f_{L,\flat}(x)=(\dot{\sf q}_{L,}f)(x)\coloneqq \langle q_{L}(x,\emparg),f\rangle_{\mathbb{T}_L^n}
$$
with the Markov kernel $q_{L}(x,v)\coloneqq L^n\,{\bf 1}_{v+Q_L}(x)$ on $\mathbb{T}^n\times\mathbb{T}^n_L$. The latter is the restriction of the Markov kernel
\begin{equation}q_L=L^n\,\sum_{v\in\mathbb{T}^n_L}{\bf 1}_{v+Q_L}\otimes {\bf 1}_{v+Q_L}\quad \text{on }\mathbb{T}^n\times\mathbb{T}^n.\end{equation}
Note that $\int_{\mathbb{T}^n}q_L(x,y)dy=1$ as well as $\int_{\mathbb{T}^n_L} q_L(x,v)dm_L(v)=1$.

The \emph{projection} from $L^2(\mathbb{T}^n)$ onto the set of piecewise constant functions on $\mathbb{T}^n$ is given by $f\mapsto f_{\flat,L}$ with
$$ f_{\flat,L}(x)\coloneqq ({\sf q}_Lf)(x)\coloneqq \langle q_L(x,\emparg),f\rangle_{\mathbb{T}^n}= L^n\,\sum_{v\in\mathbb{T}^n_L}\langle 1_{v+Q_L},f\rangle_{\mathbb{T}^n}\cdot1_{v+Q_L}(x)\fstop$$

Here and in the sequel, the \emph{integral operators} associated with kernels $p,q,r$ will be denoted by ${\sf p,q,r}$, resp.
In general, these are regarded as integral operators on $\mathbb{T}^n$.  If we want to regard them as  integral operators on $\mathbb{T}^n_L$,
 we write ${\sf \dot p,\dot q, \dot r}$ instead.

\paragraph{(b) Fourier Extension/Projection.} \label{four-proj-ext}
Let $\mathcal D_L$ denote the linear span of $\{\varphi_z: \ z\in\Z^n_L\}$.
Every  function $f$ on the discrete torus $\mathbb{T}^n_L$ can be uniquely represented as $f=\sum_{z\in\Z_L^n}\alpha_z\varphi_z$ with suitable coefficients $\alpha_z\in\R$ for $z\in \Z^n_L$, and thus uniquely \emph{extends} to a function 
$f_{L,\sharp}\in \mathcal D_L$ on the continuous torus $\mathbb{T}^n$.
Formally,
$$ f_{L,\sharp}(x)
\coloneqq (\dot{\sf r}_{L}f)(x)\coloneqq \langle f,r_{L}(x,\emparg)\rangle_{\mathbb{T}_L^n}\coloneqq\sum_{z\in\Z^n_L}\langle f,\varphi_z\rangle_{\mathbb{T}_L^n}\cdot \varphi_z(x)$$
with the kernel
\begin{equation}
r_{L}\coloneqq\sum_{z\in\Z_L^n}\varphi_z\otimes\varphi_z\quad \text{on }\mathbb{T}^n\times\mathbb{T}^n.\end{equation}
Regarded as a kernel on  $\mathbb{T}_L^n\times\mathbb{T}^n$, the latter defines the Fourier extension operator. As a kernel on $\mathbb{T}_L^n\times\mathbb{T}_L^n$ it indeed is the identity.

Conversely, the \emph{projection} from $\bigcup_{s}H^s(\mathbb{T}^n)$ onto $\mathcal D_L$ is given by $f\mapsto f_{\sharp,L}$ with
$$ f_{\sharp,L}(x)
\coloneqq ({\sf r}_{L}f)(x)\coloneqq \langle f,r_{L}(x,\emparg)\rangle_{\mathbb{T}^n}\coloneqq\sum_{z\in\Z^n_L}\langle f,\varphi_z\rangle_{\mathbb{T}^n}\,\varphi_z(x).$$
In particular, if $f=\sum_{z\in\Z^n}\alpha_z\varphi_z$ then $f_{\sharp,L}=\sum_{z\in\Z_L^n}\alpha_z\varphi_z$.

\paragraph{(c) Enhancement and Reduction.} 

For $f=\sum_{z\in\Z_L^n}\alpha_z\varphi_z\in \bigcup_{s}H^s(\mathbb{T}^n)$ we define its \emph{spectral reduction}  and its \emph{spectral enhancement}, resp., by
$$f_L^{{-\circ}}\coloneqq\sum_{z\in\Z_L^n}\left(\frac{\lambda_{L,z}}{\lambda_{z}}\right)^{n/4}\alpha_z\varphi_z, \qquad
f_L^{{+\circ}}\coloneqq\sum_{z\in\Z_L^n}\left(\frac{\lambda_z}{\lambda_{L,z}}\right)^{n/4}\alpha_z\varphi_z.$$
Note that
\begin{equation}\label{ratio-lambda}
\frac{\lambda_{L,z}}{\lambda_{z}}=\frac{L^2}{\pi^2\,|z|^2}\sum_{k=1}^n \sin^2(\pi z_k/L)\  \in \ \left[ (2/\pi)^2, 1\right]\quad\text{and }\to 1\text{ as }L\to\infty.
\end{equation}
Similarly, we define its \emph{integral reduction}  and its \emph{integral enhancement}, resp., by
$$f_L^{{\circ-}}\coloneqq\sum_{z\in\Z_L^n}\vartheta_{L,z}\alpha_z\varphi_z,\qquad f_L^{{\circ+}}\coloneqq\sum_{z\in\Z_L^n}\frac1{\vartheta_{L,z}}\alpha_z\varphi_z$$
with 
\begin{equation}\label{def-theta}
\vartheta_{L,z}\coloneqq \prod_{k=1}^n\left(\frac L{\pi z_k}\, \sin \left(\frac{\pi z_k}L\right)\right)\ \in \ \left[ (2/\pi)^n, 1\right]\quad\text{and }\to 1\text{ as }L\to\infty.
\end{equation}
In terms of integral operators this can be expressed as
$$f_L^{{-\circ}}={\sf r}_L^{{-\circ}}f, \qquad f_L^{{+\circ}}={\sf r}_L^{{+\circ}}f, \qquad f_L^{{\circ-}}={\sf r}_L^{{\circ-}}f, \qquad f_L^{{\circ+}}={\sf r}_L^{{\circ+}}f$$
with integral and enhancement kernels  on $\mathbb{T}^n\times\mathbb{T}^n$ defined as follows
$$r_L^{{+\circ}}=\sum_{z\in\Z_L^n}\left(\frac{\lambda_z}{\lambda_{L,z}}\right)^{n/4}\varphi_z\otimes\varphi_z,\qquad
r_L^{{-\circ}}=\sum_{z\in\Z_L^n}\left(\frac{\lambda_z}{\lambda_{L,z}}\right)^{-n/4}\varphi_z\otimes\varphi_z
$$
$$r_L^{{\circ+}}=\sum_{z\in\Z_L^n}\vartheta_{L,z}^{-1}\ \varphi_z\otimes\varphi_z,\qquad
r_L^{{\circ-}}=\sum_{z\in\Z_L^n}\vartheta_{L,z}\ \varphi_z\otimes\varphi_z
$$
$$r_L^{{+}}=\sum_{z\in\Z_L^n}\left(\frac{\lambda_z}{\lambda_{L,z}}\right)^{n/4}\vartheta_{L,z}^{-1}\ \varphi_z\otimes\varphi_z,\qquad
r_L^{{-}}=\sum_{z\in\Z_L^n}\left(\frac{\lambda_z}{\lambda_{L,z}}\right)^{-n/4}\vartheta_{L,z}\ \varphi_z\otimes\varphi_z
$$

\begin{lemma}\label{averages}
For $f\in  \mathcal D_L$,
$${\sf q}_Lf=f_L^{{\circ-}}\text{ on }\mathbb{T}^n_L\quad\text{and}\quad 
{\sf q}_L( f_L^{{\circ+}})=f\text{ on }\mathbb{T}^n_L.
$$
\end{lemma}

\begin{proof} For $f=\Phi_z$ with $z\in\Z^n_L$, and for $v\in\mathbb{T}^n_L$,
\begin{align*}{\sf q}_Lf(v)&=L^n\int_{v+Q_L}\Phi_z(x)dx=\Phi_z(v)\cdot L^n\,\int_{Q_L}\Phi_z(x)dx\\
&=\Phi_z(v)\cdot \prod_{k=1}^n L\,\int_{-\frac1{2L}}^{\frac1{2L}}\cos(2\pi x_k z_k)dx_k=\Phi_z(v)\cdot\prod_{k=1}^n\left(\frac L{\pi z_k}\, \sin \left(\frac{\pi z_k}L\right)\right).
\end{align*}
Therefore, for $f=\varphi_z$ with $z\in\Z^n_L$, and for $v\in\mathbb{T}^n_L$,
$${\sf q}_Lf(v)=f(v)\cdot \vartheta_{L,z}.$$
Thus 
the claim follows.
\end{proof}

\paragraph{(d) Continuous vs.~Discrete Scalar Product.}

For functions $f=\sum_{z\in\Z_L^n}\alpha_z\varphi_z$ and $g=\sum_{w\in\Z_L^n}\beta_w\varphi_w$, the scalar products in $\mathbb{T}^n_L$ and in $\mathbb{T}^n$ coincide:
\[
\langle f,g\rangle_{\mathbb{T}_L^n}=\langle f,g\rangle_{\mathbb{T}^n}=\sum_{z\in\Z_L^n}\alpha_z\,\beta_{z} \fstop
\]
This simple identity, however, no longer holds if the Fourier representation of $f$ and $g$ also contains terms with higher frequencies.

\begin{lemma} \label{high-frequ}
(i) For $f=\sum_{z\in\Z_K^n}\alpha_z\varphi_z$ and $g=\sum_{w\in\Z_K^n}\beta_w\varphi_w$, 
$$\langle f,g\rangle_{\mathbb{T}_L^n}=\sum_{z\in\Z_K^n}\sum_{w\in\Z^n, \|z+Lw\|_\infty<K/2} \alpha_z\,\beta_{z+Lw}.$$
(ii) For any $\alpha: \Z^{n}\to \mathbb R$, the limit $f =\displaystyle\sum^{\Box}_{z\in\Z^n}\alpha_z\varphi_z $ exists in $L^2(\mathbb{T}^n_L)$ if and only if 
  \begin{equation}\label{L2-discrete-torus}
  \sup_{K}\Big\| \sum_{z\in\Z_K^n}\alpha_z\varphi_z \Big\|^2_{\mathbb{T}^n_L}<\infty
\end{equation}
(iii)
For all $f=\displaystyle\sum^{\Box}_{z\in\Z^n}\alpha_z\varphi_z$ and $g=\displaystyle\sum^{\Box}_{w\in\Z^n}\beta_w\varphi_w$ in $L^2(\mathbb{T}^n_L)$,
\[
\langle f,g\rangle_{\mathbb{T}_L^n}
= \lim_{K \to \infty} \sum_{z \in \mathbb{Z}_{K}^{n}} \sum_{w \in \mathbb{Z}^{n} : \norm{z+Lw}_{\infty} < K/2} \alpha_{z} \beta_{z + Lw}
\comma \qquad
\langle f,g\rangle_{\mathbb{T}^n}=\sum^{\Box}_{z\in\Z^n}\alpha_z\,\beta_{z}.
\]

\end{lemma}
\begin{proof}
We first prove (i).
To that extent, we prove the analogous assertion in the complex Hilbert space:
for all $f=\sum_{z\in\Z_K^n} a_z\Phi_z$ and $g=\sum_{w\in\Z_K^n} b_w\Phi_w$,
\begin{align*}
\langle f,g\rangle_{\mathbb{T}_L^n}&=
\Big\langle \sum_{z\in\Z^n_K}a_z\Phi_z, \sum_{w\in\Z^n_K} b_w\Phi_w\Big\rangle_{\mathbb{T}_L^n}
\\
&=\int_{\mathbb{T}^n_L}\sum_{z\in\Z^n_K}\sum_{w\in\Z^n_K} a_z\,\overline{b}_{w}
\exp\big( 2\pi i\, v(z-w) \big)\,dm_L(v)
\\
&=\sum_{z\in\Z_K^n}\sum_{w\in\Z_K^n} a_z\, \overline{b}_{w}\cdot
\int_{\mathbb{T}^n_L}
\exp\big( 2\pi i\, v(z-w) \big)\,dm_L(v)
\\
&=\sum_{z\in\Z_K^n}\sum_{w\in\Z^n, \|z+Lw\|_\infty<K/2} a_z\,\overline{b}_{z+Lw}\end{align*}
since for every $z\in\Z^n$
\[
\int_{\mathbb{T}^n_L}
\exp(2\pi i\, v z)\,dm_L(v)=
\begin{cases}1,& \ \text{if }z\in L\Z^n
\\
0,&\ \text{else}
\end{cases} \fstop
\]
The claim for the real Hilbert space then follows choosing $a_z=\frac1{\sqrt 2}(\alpha_z+i\alpha_{-z})$ and $a_{-z}=\frac1{\sqrt 2}(\alpha_z-i\alpha_{-z})$ for $z\in\hat\Z^n$ and analogously $b_z$.

We prove (ii).
Assume first that $f \in L^{2}(\mathbb{T}^{n}_{L})$.
Then $\sum_{z \in \mathbb{Z}^{n}_{K}} \alpha_{z} \varphi_{z}$ converges to $f$ in $L^{2}(\mathbb{T}_{L}^{n})$.
This implies \eqref{L2-discrete-torus}.

Conversely, assume that \eqref{L2-discrete-torus} holds.
Then a martingale argument similar to that of Lemma \ref{green-kernel-series} shows that $f$ is the limit in $L^{2}(\mathbb{T}^{n}_{L})$ of $\sum_{z \in \mathbb{Z}_{K}^{n}} \alpha_{z} \varphi_{z}$.

We prove (iii).
We only need to show the first equation.
By linearity it is sufficient to show it for $f = g$.
In view of what precedes, we have
\begin{equation*}
  \norm{f}_{\mathbb{T}_{L}^{n}}^{2} = \lim_{K \to \infty} \norm{\sum_{z \in \mathbb{Z}_{K}^{n}} \alpha_{z} \varphi_{z}}^{2} = \lim_{K \to \infty} \sum_{z \in \mathbb{Z}_{K}^{n}} \sum_{w \in \mathbb{Z}^{n} : \norm{z+Lw}_{\infty} < K/2} \alpha_{z} \alpha_{z + Lw},
\end{equation*}
which proves the claim.
The convergence of the series is ensured by \eqref{L2-discrete-torus}.
\end{proof}

\begin{remark}
According to the previous lemma, in particular, for every 
$f=\displaystyle\sum^{\Box}_{z\in\Z^n}\alpha_z\varphi_z$,
$$\big\|f\big\|^2_{\mathbb{T}_L^n}=\sum_{z\in\Z^n}\sum_{w\in\Z^n}\alpha_z\,\alpha_{z+Lw}$$
if the latter series is absolutely convergent.

One can show (cf.~proof of Theorem \ref{thm-hL}) that the latter series is absolutely convergent if $f\in \bigcup_{s>n/2}H^s(\mathbb{T}^n)$. This is in accordance with the Sobolev Embedding Theorem which asserts that in this case $f\in\mathcal C(\mathbb{T}^n)$ and thus guarantees that the pointwise evaluation of $f$ (at the lattice points of $\mathbb{T}^n_L$) is meaningful. 
\end{remark}

\section{The Polyharmonic Gaussian Field on the Discrete Torus}

 \subsection{Definition and Construction of the Field} Throughout the sequel, fix integers $n$ and $L$. For convenience, we assume that~$L$ is odd, and we set~$N\coloneqq L^n$. 

 \begin{definition}  A random field $h_L=(h_L(v))_{v\in\mathbb{T}^n_L}$ --- defined on some probability space $(\Omega,\mathfrak A, \mathbf P)$ ---  is called \emph{polyharmonic Gaussian field on the discrete $n$-torus} if it is a centered Gaussian field with covariance function $k_L$ (see~\eqref{kL}).
 \end{definition}

 \begin{proposition}\label{discrete-iid}
 Given iid standard normals $\xi_z$ for ${z\in\Z_L^n\setminus\{0\}}$ on a probability space $(\Omega,\mathfrak A, \mathbf P)$, a polyharmonic Gaussian field on $\mathbb{T}^n_L$ is defined by
\begin{equation}\label{haL}
 h^\omega_L(v)\coloneqq \frac1{\sqrt{a_n}}\sum_{z\in\Z_L^n\setminus\{0\}} 
\frac1{\lambda_{L,z}^{n/4}}
\cdot 
\xi_z^\omega\cdot \varphi_z(v) \qquad (\forall v\in\mathbb{T}_L^n, \ \omega\in\Omega)\fstop
\end{equation}
\end{proposition}
Here  $\lambda_{L,z}=4L^2\,\sum_{k=1}^n \sin^2\big(\pi z_k/L\big)$ for $z\in\Z_L^n$ are the  eigenvalues of the discrete Laplacian, see~\eqref{lL}, and the eigenvalue 0 is excluded in the representation of the random field.

\begin{proof} For all $v,u\in \mathbb{T}_L^n$,
\begin{align*}\mathbf E\big[h_L(v)\,h_L(u)\big]= \frac1{a_n} 
\sum_{z\in\Z_L^n\setminus\{0\}} \frac1{\lambda_{L,z}^{n/2}}\cdot\varphi_z(v)\,\varphi_z(u)
=k_L(v,u) \fstop&  \qedhere
\end{align*}
\end{proof}


%

Alternatively   the polyharmonic field on the discrete torus $\mathbb{T}^n_L$ can be defined in terms of the white noise on the discrete torus. Recall that a random field $\Xi= (\Xi_v)_{v\in\mathbb{T}^n_L}$ 
is called \emph{white noise on $(\mathbb{T}^n_L,m_L)$} if the  $\Xi_v$ for ${v\in\mathbb{T}_L^n}$ are independent centered Gaussian random variables with variance $L^n$. (This normalization guarantees that $\int_{\mathbb{T}^n_L}\Xi_vdm_L(v)$ is $\mathcal N(0,1)$ distributed.)

\begin{proposition}
Given a white noise $\Xi= (\Xi_v)_{v\in\mathbb{T}^n_L}$ on $(\mathbb{T}^n_L,m_L)$, a polyharmonic Gaussian field on $\mathbb{T}^n_L$ is defined by
\begin{align*}h^\omega_L(v)&
 =\frac1{\sqrt{a_n}\,L^n}\,\sum_{u\in\mathbb{T}^n_L}\mathring G_L^{n/4}(v,u)\, \Xi^\omega_u
\end{align*}
with 
$$\mathring G_L^{n/4}(v,u)=\sum_{z\in\Z_L^n\setminus\{0\}} 
\frac1{\lambda_{L,z}^{n/4}}\cdot \cos\big(2\pi z(v-u)\big)\fstop$$
\end{proposition}

\begin{proof} For all $v,w\in\mathbb{T}^n_L$,
\begin{align*}
\mathbf E\big[h_L(v)\,h_L(w)\big]&= \frac1{a_n\, L^{2n}} \,\sum_{u\in\mathbb{T}^n_L}\mathring G_L^{n/4}(v,u)
\cdot \mathring G_L^{n/4}(w,u)
\cdot L^n\\
&=\frac1{a_n}\mathring G_L^{n/2}(v,w)=k_L(v,w)\fstop \qedhere
\end{align*}
\end{proof}

In other words,
\begin{equation}\label{bijection}
h^\omega_L=\frac1{\sqrt{a_n}}\,\mathring{\sf G}_L^{n/4} \, \Xi^\omega
\end{equation}
with $\Xi\coloneqq (\Xi_v)_{v\in\mathbb{T}^n_L}$ being a white noise on $\mathbb{T}^n_L$.
The latter is a  Gaussian random variable on $\R^N\cong\R^{\mathbb{T}_L^n}$ --- recall that $N=L^n$ ---  with distribution 
\[
d{\mathbf P}(\Xi)=\frac1{(2\pi N)^{N/2}}\,e^{-\frac1{2N}\|\Xi\|^2}\,d\Leb^N(\Xi) \fstop
\]
Here $\|\Xi \|$ denotes the Euclidean norm of $\Xi\in\R^N$, and thus under the identification $\R^{N}\cong\R^{\mathbb{T}_L^n}$,
$$\frac1N \|\Xi \|^2=\|\Xi \|^2_{L^2(\mathbb{T}^n_L,m_L)}\fstop$$

\subsection{A Second Look on the Polyharmonic Gaussian Field on Discrete Tori}
\paragraph{(a)}
Consider the 
orthogonal decomposition of $\R^N$ into the line $\R\cdot (1,\ldots,1)$ and its orthogonal complement $\mathring{H}\coloneqq \{\Xi\in\R^N: \sum_{v=1}^N\Xi_v=0\}$. More precisely, consider  the maps
\[
\bar A: \R^N\longrightarrow\R, \quad \Xi\longmapsto \bar \Xi\coloneqq \frac1{\sqrt N}\sum_{v=1}^{N} \Xi_v
\]
and
\[
\mathring A: \R^N\longrightarrow \mathring{H}, \quad \Xi\longmapsto  \mathring \Xi \quad \text{with} \quad \mathring \Xi_j\coloneqq \Xi_j-\frac1{\sqrt N} \bar \Xi\fstop
\]
Note that $A\coloneqq (\mathring A,\bar A): \R^N\to \mathring{H}\times\R\subset  \R^{1+N}$ is a bijective linear map with $A^T\,A=E_N$ and inverse given by
\[
B:\mathring{H}\times\R\to \R^N, \quad (\mathring \Xi,t)\mapsto \mathring \Xi+\frac t{\sqrt N}\cdot (1,\ldots,1) \fstop
\]
Thus if $\Leb^{N-1}_H$ denotes the $(N-1)$-dimensional Lebesgue measure on the hyperplane $\mathring{H}$ then on $\R^N$,
\[
\Leb^N=\Leb_{\mathring{H}}^{N-1}\otimes\Leb^1\fstop
\]

The push forward $\bar A_*{\mathbf P}$ is the  normal distribution $\mathcal N(0,\sqrt N)$ on the real line.
The push forward~$\mathring{\mathbf P}\coloneqq \mathring A_*{\mathbf P}$, called ``law of the grounded white noise", is a Gaussian measure on the hyperplane $\mathring{H}$ given explicitly as
\[
d\mathring{\mathbf P}(\Xi)=
\frac1{(2\pi N)^{\frac{N-1}2}}\,e^{-\frac{1}{2N} \|\Xi \|^2}\,d\Leb^{N-1}_{\mathring{H}}(\Xi) \fstop
\]
It can also be characterized as the conditional law $
\mathbf{P}(\emparg | \bar A=0)$. 

\paragraph{(b)}
Let us define a measure on $\R^{N}\cong\R^{\mathbb{T}_L^n}$
by
\begin{equation}
d\widehat{\boldsymbol\nu}(h)\coloneqq c_n\,e^{-\frac{a_n}{2N}\|(-\Delta_L)^{n/4}h\|^2}\,d\Leb^N(h)
\end{equation}
with a constant 
\begin{equation}
c_n\coloneqq \left(\frac{a_n}{2\pi N}\right)^{\frac{N-1}2}\cdot
\prod_{z\in\Z^n_L\setminus\{0\}}\left(
4L^2\,\sum_{k=1}^n \sin^2\big(\pi z_k/L\big)\right)^{n/4}\comma
\end{equation}
 and consider the push forwards under the maps $\bar A$ and $\mathring A$ introduced above.
Then
\[
\bar A_*\widehat{\boldsymbol\nu}=c_n\,\Leb^1\quad\text{on }\R^1
\]
and
$\boldsymbol\nu\coloneqq \mathring A_*\widehat{\boldsymbol\nu}$ is a measure (actually, a probability measure as we will see below) on the hyperplane $\mathring{H}\coloneqq \{\Xi\in\R^{N}: \bar\Xi=0\}\cong\R^{N-1}$ given by
\[
d{\boldsymbol\nu}(h)\coloneqq c_n\,e^{-\frac{a_n}{2N}\|(-\Delta_L)^{n/4}h\|^2}\,d\Leb^{N-1}_{\mathring{H}}(h) \fstop
\]
\paragraph{(c)}
Now consider the map
\[
T: \R^{N}\to\R^{N}\comma\qquad h\mapsto \Xi=\sqrt{a_n}\, (-\Delta_L)^{n/4}h
\]
as well as its restriction $\mathring T: \mathring{H}\to \mathring{H}$. The latter is bijective with inverse 
\[
\mathring T^{-1}: \mathring{H}\to \mathring{H}\comma\qquad \Xi\mapsto h=\frac1{\sqrt{a_n}}\, \mathring{\sf G}_L^{n/4}\Xi\comma
\]
cf. \eqref{bijection}, and with determinant
\[
\mathrm{det}\, \mathring T=a_n^{\frac{N-1}2}\cdot\prod_{z\in\Z^n_L\setminus\{0\}}\lambda_{L,z}^{n/4} \fstop
\]

\begin{theorem}\label{repres} The distribution of the discrete polyharmonic field on ${\mathbb{T}_L^n}$ is given by the probability measure $\boldsymbol\nu$  on $\R^{\mathbb{T}_L^n}\cong\R^N$. (Indeed, it is supported there by the hyperplane of grounded fields.) Furthermore,
\[
\mathring T_*\boldsymbol\nu=\mathring{\mathbf P} \fstop
\]
\end{theorem}

\begin{proof} For bounded measurable $f$ on $\mathring{H}$,

\begin{align*}
\int_{\mathring{H}} f(\Xi)\,d\mathring T_*\boldsymbol\nu(\Xi)&=\int_{\mathring{H}}f(\mathring T h)\,d\boldsymbol\nu(h)\\
&=
c_n\,\int_{\mathring{H}}f(\mathring T h)\,
e^{-\frac{a_n}{2N}\|(-\Delta_L)^{n/4}h\|^2}\,d\Leb^{N-1}_{\mathring{H}}(h)\\
&=
c_n\,\int_{\mathring{H}}f(\mathring T h)\,
e^{-\frac{1}{2N}\| \mathring T h\|^2}\,d\Leb^{N-1}_{\mathring{H}}(h)\\
&=
c_n\,\mathrm{det}\, \mathring T^{-1}\,\int_{\mathring{H}}f(\Xi)\,
e^{-\frac{1}{2N}\| \Xi\|^2}\,d\Leb^{N-1}_{\mathring{H}}(\Xi)\\
&=
c_n\,\mathrm{det}\, \mathring T^{-1}\,(2\pi N)^{\frac{N-1}2}\,\int_Hf(\Xi)\,d\mathring{\mathbf P}(\Xi).
\end{align*}
Since $c_n\,\mathrm{det}\, \mathring T^{-1}\,(2\pi N)^{\frac{N-1}{2}}=1$ according to our choice of $c_n$,
this proves the claim.
\end{proof} 

\subsection{The Reduced Polyharmonic Gaussian Field on the Discrete Torus}\label{ss:ReducedFields}

Besides the polyharmonic Gaussian field $h_L$ on the discrete torus,  we occasionally consider two closely related random fields $h_L^{{-\circ}}$ and $h_L^{{-}}$ in the defining properties of which the eigenvalues $\lambda_{L,z}=4L^2\,\sum_{k=1}^n \sin^2\big(\pi z_k/L\big)$  of the discrete Laplacian are replaced by the eigenvalues $\lambda_z=(2\pi|z|)^2$ of the continuous Laplacian or by $\lambda_z\cdot \vartheta_{L,z}^{-4/n}$, resp., with $\vartheta_{L,z}$ as in \eqref{def-theta}.

More precisely, a  \emph{spectrally reduced discrete polyharmonic Gaussian field}   is a centered Gaussian field $h_L^{{-\circ}}=(h^{{-\circ}}_L(v))_{v\in\mathbb{T}^n_L}$ with covariance function 
 \begin{align}k^{{-\circ}}_L(v,u)&\coloneqq \frac1{a_n}\sum_{z\in\Z^n_L\setminus\{0\}} \frac1{\lambda^{n/2}_{z}}\ \varphi_z(v)\, \varphi_z(u)
\nonumber
=\frac1{a_n}\sum_{z\in\Z^n_L\setminus\{0\}} \frac1{(2\pi|z|)^{n}}\cdot 
\cos\Big(2\pi \, z\cdot (v-u)\Big)\fstop
\end{align}
A  \emph{reduced discrete polyharmonic Gaussian field}   is a centered Gaussian field $h_L^{{-}}=(h^{{-}}_L(v))_{v\in\mathbb{T}^n_L}$ with covariance function 
 \begin{align}k^{{-}}_L(v,u)&\coloneqq \frac1{a_n}\sum_{z\in\Z^n_L\setminus\{0\}} \frac{\vartheta_{L,z}^2}{\lambda^{n/2}_{z}}\ \varphi_z(v)\, \varphi_z(u)
\nonumber
=\frac1{a_n}\sum_{z\in\Z^n_L\setminus\{0\}} \frac{\vartheta^2_{L,z}}{(2\pi|z|)^{n}}\cdot 
\cos\Big(2\pi \, z\cdot (v-u)\Big)\fstop
\end{align}


Similarly as before for $h_L$, we obtain the following representation results.
 \begin{remark} 
  (i) Given  a polyharmonic Gaussian field $h_L$ on $\mathbb{T}^n_L$,  a reduced polyharmonic Gaussian field and a spectrally  reduced polyharmonic Gaussian field  on $\mathbb{T}^n_L$ are defined by
\begin{align*}
h_L^{{-}}\coloneqq {\sf r}_L^{{-}}(h_L), \qquad h_L^{{-\circ}}\coloneqq {\sf r}_L^{{-\circ}}(h_L)
\fstop\end{align*}

 (ii)
 Given iid standard normals $\xi_z$ for ${z\in\Z_L^n\setminus\{0\}}$, a reduced polyharmonic Gaussian field on $\mathbb{T}^n_L$ is defined by
\begin{equation}
  h_L^{{-}}(v)\coloneqq \frac1{\sqrt{a_n}}\sum_{z\in\Z_L^n\setminus\{0\}} 
\frac{\vartheta_{L,z}}{\lambda_{z}^{n/4}}
\cdot 
\xi_z\cdot \varphi_z(v) \fstop
\end{equation}
and  a spectrally reduced polyharmonic Gaussian field by
\begin{equation}
 h^{{-\circ}}_L(v)\coloneqq \frac1{\sqrt{a_n}}\sum_{z\in\Z_L^n\setminus\{0\}} 
\frac1{\lambda_{z}^{n/4}}
\cdot 
\xi_z\cdot \varphi_z(v) \fstop
\end{equation}

\end{remark}

\begin{figure}[htb!]
     \centering
     \begin{subfigure}[b]{0.3\textwidth}
         \centering
         \includegraphics[width=\textwidth]{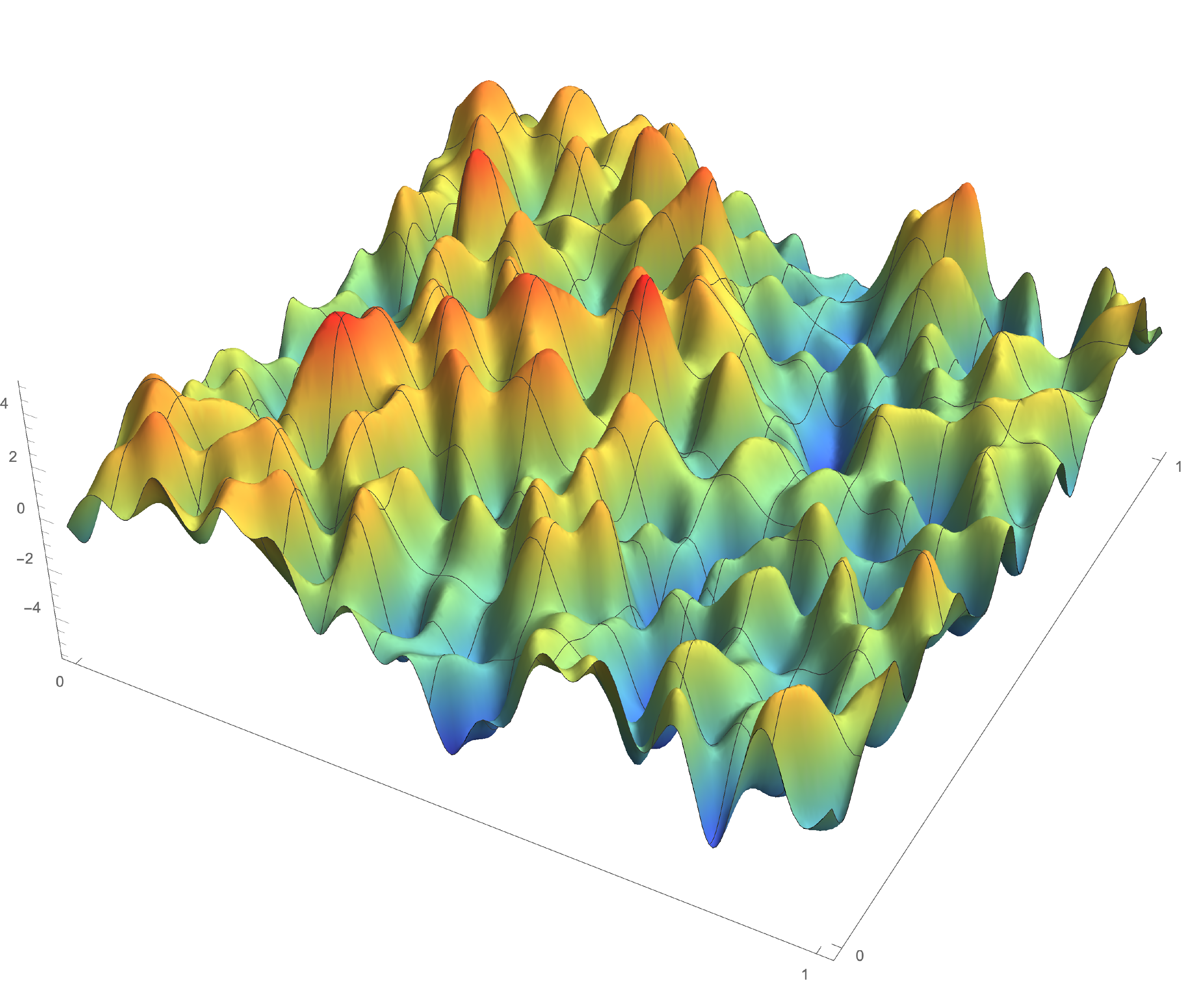}
         \caption{$h_{25,\sharp}$}
     \end{subfigure}
     \hfill
     \begin{subfigure}[b]{0.3\textwidth}
         \centering
         \includegraphics[width=\textwidth]{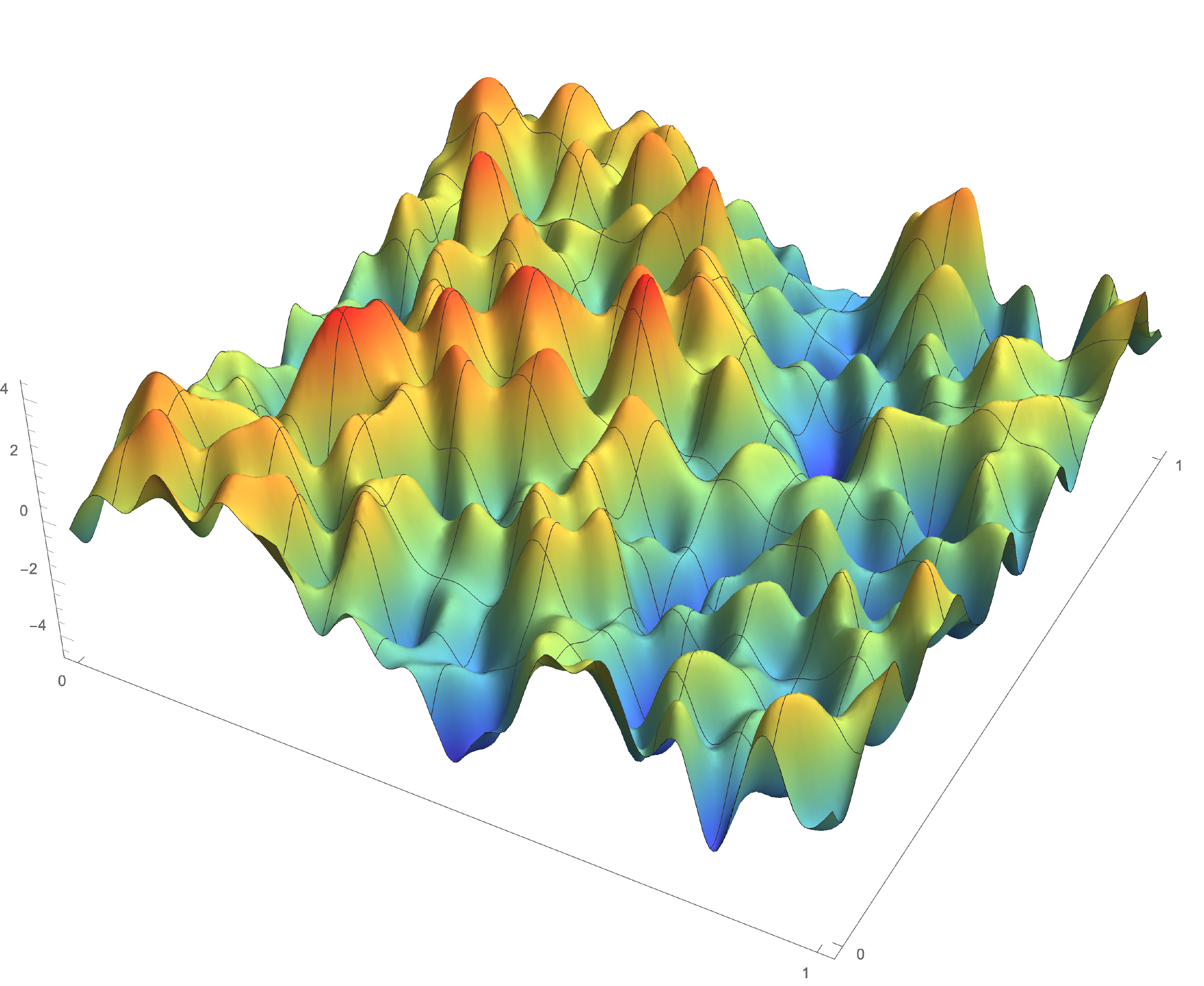}
         \caption{$h_{\sharp,25}$}
     \end{subfigure}
     \hfill
     \begin{subfigure}[b]{0.3\textwidth}
         \centering
       \includegraphics[width=\textwidth]{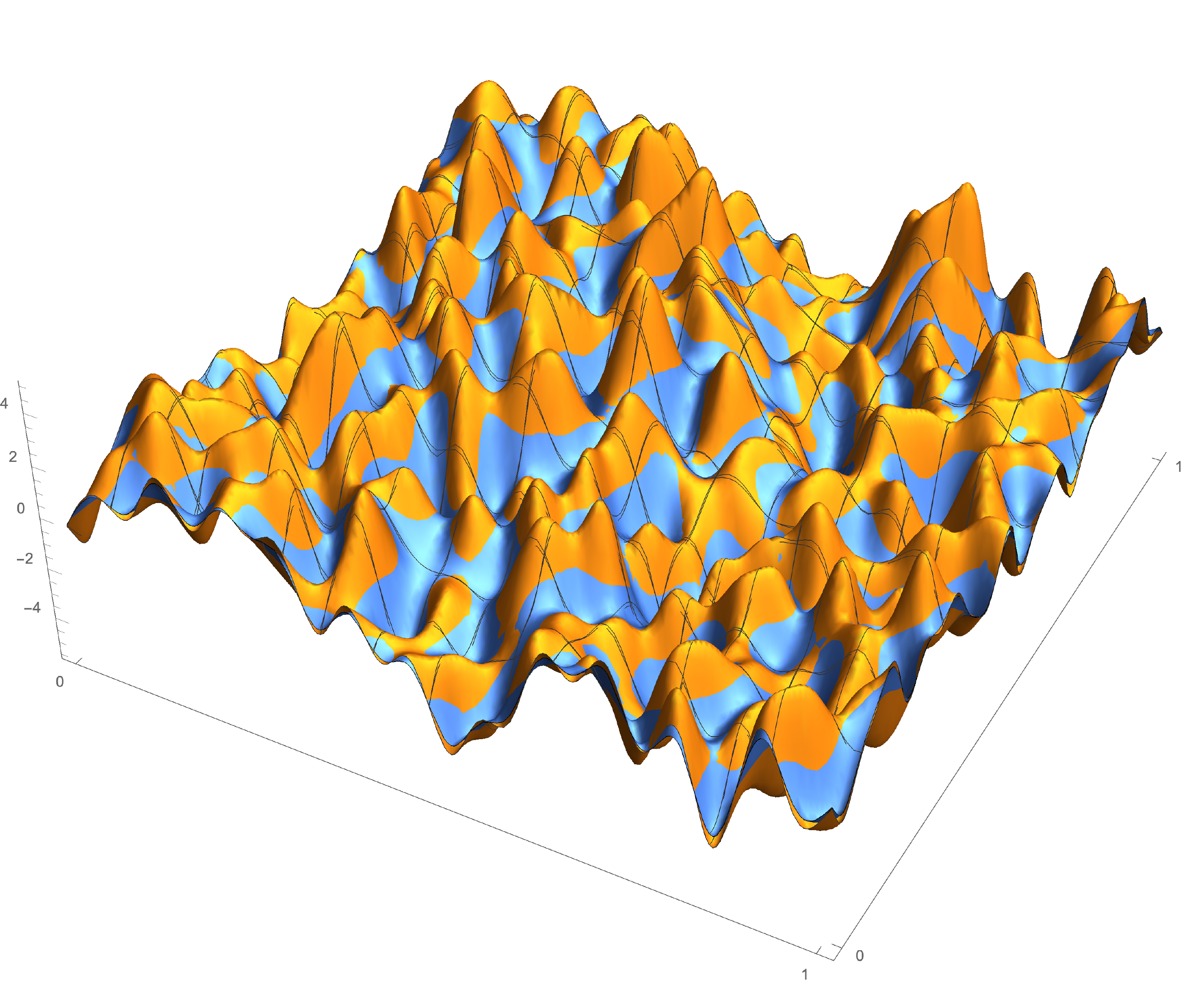}
         \caption{$h_{25,\sharp}$ (blue) and $h_{\sharp,25}$ (orange)}
     \end{subfigure}
        \caption{Fourier extension/projection of~$h$ on~$\mathcal{D}_{25}$ with same realization of the randomness.}
\end{figure}

\section{The Polyharmonic Gaussian Field on the Continuous Torus and its (Semi-) Discrete Approximations}

This section is devoted to the analysis of approximation properties for the polyharmonic field on the continuous torus in terms of Gaussian fields on the discrete torus and semi-discrete extensions of the latter on the continuous torus.
\begin{itemize}
\item
The basic objects are the polyharmonic field $h$ on the continuous torus and its discrete counterpart, the polyharmonic field $h_L$ on the discrete torus.
\item Starting from the field $h$ on $\mathbb{T}^n$, we define its Fourier projection (aka eigenfunction approximation) $h_{\sharp,L}$, its piecewise constant projection $h_{\flat,L }$, its natural projection $h_{\circ,L }$, and
its enhanced projection $h_{+,L}$. 
All of them are 
 Gaussian random fields on $\mathbb{T}^n$.
\item Starting from the field $h_L$ on $\mathbb{T}^n_L$, we define its Fourier extension $h_{L,\sharp}$ and its piecewise constant extension $h_{L,\flat}$. 
Analogous extensions are defined for the so-called spectrally reduced discrete field $h_L^{{-\circ}}$ and the reduced discrete field $h_L^{{-}}$ on $\mathbb{T}^n_L$. 
All these extensions are Gaussian random fields on $\mathbb{T}^n$.
\end{itemize}
To summarize
\begin{itemize}[-]
\item $\flat$ stands for piecewise constant extension/projection, $\sharp$ for Fourier extension/restriction;
\item $h_{L,\ast}$ with $\ast\in\{\flat,\sharp\}$ denotes the respective extension of the discrete field $h_L$; similarly for $h_L^{{-\circ}}$ and $h_L^{{-}}$;
\item $h_{\ast, L}$ with $\ast\in\{\flat,\sharp, \circ, +\}$ denotes the projection of the continuous field $h$ onto the respective class of fields of order $L$ on the continuous torus.
\end{itemize}
%
%
%

\subsection{The Polyharmonic Gaussian Field on the Continuous Torus and Convergence Properties of its Projections}\label{ss:Projections}

\subsubsection{The Polyharmonic Gaussian Field $h$ on the Continuous Torus}\label{ss:Projections:A}

\begin{definition} A random field $h=(\langle h|f\rangle)_{f\in H^{n/2}(\mathbb{T}^n)}$ on the continuous $n$-torus is called \emph{polyharmonic Gaussian field} if it is a centered Gaussian field with covariance function $k$ 
in the sense that
\[
\mathbf E\big[\langle h|f\rangle\cdot \langle h|g\rangle\big]=\int_{\mathbb{T}^n}\int_{\mathbb{T}^n} f(x) k(x,y) g(y)\,dy\,dx
\qquad(\forall f,g\in H^{n/2}(\mathbb{T}^n))\fstop
\]
\end{definition}

\begin{proposition}
\begin{enumerate}[$(i)$] 
\item The polyharmonic Gaussian field exists.

\item It can be realized in $\mathring H^{-\epsilon}(\mathbb{T}^n)$. 

\item The pairing $\langle h|f\rangle$ continuously extends to all $f\in \mathring H^{-n/2}(\mathbb{T}^n)$.
\end{enumerate}
\end{proposition}

\begin{proof} For even $n$, the polyharmonic field to be considered here  is just a particular case of the co-polyharmonic field considered in \cite{DHKS} on large classes of Riemannian manifolds.
For flat spaces like the torus, the arguments for proving Thm.~3.1 and Rmks.~3.4+3.5 there obviously also apply to odd $n$.
\end{proof}

\subsubsection{Fourier Projection of $h$}\label{ss:Projections:B} 
Given a polyharmonic Gaussian field $h$ on the continuous torus, we define its projection onto the space~$\mathcal D_L$ 
(see subsection \ref{four-proj-ext}(a))
by
\begin{equation}\label{haLtil}
 h_{\sharp,L}(x)\coloneqq 
 \langle h| r_{L}(x,\emparg)\rangle, \qquad r_{L}\coloneqq \sum_{z\in\Z_L^n\setminus\{0\}} \varphi_z\otimes \varphi_z
\end{equation}
This is a centered Gaussian field on $\mathbb{T}^n$ with covariance function
\begin{equation}\label{kaLtil}
k_{\sharp,L}(x,y)\coloneqq\frac1{a_n}\sum_{z\in\Z^n_L\setminus\{0\}} \frac1{\lambda_{z}^{n/2}}\cdot 
\cos\Big(2\pi \, z\cdot (x-y)\Big) 
\end{equation}
where $\lambda_{z}=(2\pi |z|)^2$.
\begin{proposition}[{\cite{DHKS}}] \label{thm-four-L}
For all $f\in \mathring H^{-n/2}(\mathbb{T}^n)$,
\[
\langle h_{\sharp,L},f\rangle_{\mathbb{T}^n} \to \langle h,f\rangle_{\mathbb{T}^n} \quad \text{in $L^2(\mathbf P)$ as }L\to\infty 
\]
and for every $\epsilon>0$,
\[h_{\sharp,L}\to h \quad \text{in $L^2(H^{-\epsilon}(\mathbb{T}^n),\mathbf P)$ as }L\to\infty \fstop\]
\end{proposition}

\begin{proof} For the convenience of the reader, we summarize briefly the argument from \cite{DHKS} for the first assertion:
\begin{align*}
\mathbf E\Big[\langle h-h_{\sharp,L},f\rangle^2\Big]&=\frac1{a_n}
\mathbf E\left[\left\lvert\sum_{z\in\Z^n\setminus\Z_L^n} \frac1{{(2\pi|z|)}^{n/2}}\,\xi_z\, \langle\varphi_z,f\rangle\right\rvert^2
\right]\\
&=\frac1{a_n}\sum_{z\in\Z^n\setminus\Z_L^n} \frac1{{(2\pi|z|)}^{n}}\, \langle\varphi_z,f\rangle^2 \to 0
\end{align*}
as $L\to\infty$ since
\[
\sum_{z\in\Z^n\setminus\{0\}} \frac1{{(2\pi|z|)}^{n}}\, \langle\varphi_z,f\rangle^2 =\big\|(-\Delta)^{-n/2}f\big\|_{L^2}^2=\big\|f\big\|_{\mathring H^{-n/2}}^2<\infty \fstop \qedhere
\]
\end{proof}

\subsubsection{Piecewise Constant 
Projection of $h$}\label{ss:Projections:C}
Given a polyharmonic Gaussian field $h$ on the continuous torus, we define its piecewise constant projection 
(cf.~subsection~\ref{cons-proj-ext}(b)) by 
\begin{equation}\label{h-flat-def}
 h_{\flat,L}(x)\coloneqq 
 \langle h| q_L(x,\emparg)\rangle, \qquad 
 q_L\coloneqq L^n\sum_{v\in\mathbb{T}^n_L}{\bf 1}_{v+Q_L}\otimes {\bf 1}_{v+Q_L}\fstop
\end{equation}
Then $(h_{\flat,L}(x))_{x\in\mathbb{T}^n}$ is a centered Gaussian field with covariance function
\begin{equation}\label{k-flat-def}
k_{\flat,L}(x,y)\coloneqq \mathbb E\big[h_{\flat,L}(x)\,h_{\flat,L}(y)\big]=L^{2n}\sum_{v,w\in\mathbb{T}^n_L}
{\bf 1}_{v+Q_L}(x)\, {\bf 1}_{w+Q_L}(y)\, \int_{v+Q_L}\int_{w+Q_L} k(x',y')\,dy'\,dx'\fstop
\end{equation}


\begin{proposition} For all $f\in L^2(\mathbb{T}^n)$,
\[
\langle h_{\flat,L},f\rangle_{\mathbb{T}^n} \to \langle h,f\rangle_{\mathbb{T}^n} \quad \text{in $L^2(\mathbf P)$ as }L\to\infty 
\]
and for every $s>n/2$,
\begin{equation*}
\norm{h_{\flat,L} - h}_{H^{-s}} \to0 \quad \text{in $L^2(\mathbf P)$ as }L\to\infty\fstop
\end{equation*}
\end{proposition}
\begin{proof}  
Since $\langle h_{\flat,L},f\rangle_{\mathbb{T}^n} =\langle h,{\sf q}_L f\rangle_{\mathbb{T}^n}$ with $({\sf q}_Lf)(x)\coloneqq \int_{\mathbb{T}^n} q_L(x,y)f(y)\,dy$, we obtain
\begin{align*}
\mathbf E\Big[\langle h_{\flat,L}- h,f\rangle^2\Big]
&=\mathbf E\Big[\langle h,{\sf q}_Lf-f\rangle^2\Big]=\big\|{\sf q}_Lf-f\big\|^2_{H^{-n/2}}\\
&\le \big\|{\sf q}_Lf-f\big\|^2_{L^2}\to0\quad\text{as $L\to\infty$.} 
\end{align*}
To prove the second assertion, let $h$ be given as
\begin{equation*}
  h = \frac{1}{\sqrt{a_{n}}} \sum_{z \in \mathbb{Z}^{n} \setminus \{0\}} \xi_{z} \frac{\varphi_{z}}{\lambda_{z}^{n/4}}
\end{equation*}
from which we get
\begin{equation*}
  h_{\flat,L} = \frac{1}{\sqrt{a_{n}}} \sum_{z \in \mathbb{Z}^{n} \setminus \{0\}}  \xi_{z} \frac{\mathsf{q}_{L} \varphi_{z}}{\lambda_{z}^{n/4}} = \frac{1}{\sqrt{a_{n}}} \sum_{z \in \mathbb{Z}^{n} \setminus \{0\}} \frac{\xi_{z}}{\lambda_{z}^{n/4}} \sum_{w \in \mathbb{Z}^{n} \setminus \{0\}} \hsp{\varphi_{w}}{\mathsf{q}_{L} \varphi_{z}} \varphi_{w}.
\end{equation*}
Consequently
\begin{equation*}
  \begin{split}
    \sqrt{a_{n}} \paren{-\Delta}^{-s/2} (h - h_{\flat,L}) &= \sum_{z \in \mathbb{Z}^{n} \setminus \{0\}} \xi_{z} \lambda_{z}^{-n/4-s/2} \paren{ \varphi_{z} - \hsp{\varphi_{z}}{\mathsf{q}_{L}\varphi_{z}}\varphi_{z} } \\
                                             &- \sum_{z \in \mathbb{Z}^{n} \setminus \{0\}} \xi_{z} \lambda_{z}^{-n/4} \sum_{w \in \mathbb{Z}^{n} \setminus \{0,z\}} \lambda_{w}^{-s/2} \hsp{\varphi_{w}}{\mathsf{q}_{L} \varphi_{z}} \varphi_{w}.
  \end{split}
\end{equation*}
Finally,
\begin{equation}\label{17}
  \begin{split}
  a_{n} \mathbf{E} \bracket[\big]{ \norm{ h - h_{\flat,L}}_{H^{-s}}^{2} } &=  \sum_{z \in \mathbb{Z}^{n} \setminus \{0\}} \lambda_{z}^{-n/2-s} \paren{1 - \hsp{\varphi_{z}}{\mathsf{q}_{L}\varphi_{z}}}^{2} \\
                                                                              &+ \sum_{z \in \mathbb{Z}^{n} \setminus \{0\}} \lambda_{z}^{-n/2} \sum_{w \in \mathbb{Z}^{n} \setminus \{0,z\}} \lambda_{w}^{-s} \hsp{\varphi_{w}}{\mathsf{q}_{L}\varphi_{z}}^{2}.
\end{split}
\end{equation}
Since $|\hsp{\varphi_{z}}{\mathsf{q}_{L}\varphi_{z}}| \le 1$ for all $L$ and 
$\hsp{\varphi_{z}}{\mathsf{q}_{L}\varphi_{z}} \to 1$ as $L \to \infty$ and 
\begin{equation*}
  \sum_{z \in \mathbb{Z}^{n} \setminus \{0\}} \lambda_{z}^{-n/2-s} \paren{ 1 - \hsp{\varphi_{z}}{\mathsf{q}_{L}\varphi_{z}} }^{2} \leq \sum_{z \in \mathbb{Z}^{n} \setminus \{0\}} \lambda_{z}^{-n/2-s} < \infty,
\end{equation*}
we find that the first term on the RHS of \eqref{17} vanishes as $L \to \infty$.
By Parseval's identity and the fact that $\lambda_z\ge1$ for all $z\not=0$, we get that
\begin{align*}
\sum_{z \in \mathbb{Z}^{n} \setminus \{0\}} \lambda_{z}^{-n/2} \sum_{w \in \mathbb{Z}^{n} \setminus \{0,z\}} \lambda_{w}^{-s} \hsp{\varphi_{w}}{\mathsf{q}_{L}\varphi_{z}}^{2}
&\le \sum_{w \in \mathbb{Z}^{n} \setminus \{0\}} \lambda_{w}^{-s} \sum_{z \in \mathbb{Z}^{n} \setminus \{0,w\}} \hsp{\varphi_{w}}{\mathsf{q}_{L}\varphi_{z}}^{2}\\
&=\sum_{w \in \mathbb{Z}^{n} \setminus \{0\}} \lambda_{w}^{-s} \,\left[\norm{\mathsf{q}_{L}\varphi_{w}}^2_{L^{2}}-\hsp{\varphi_{w}}{\mathsf{q}_{L}\varphi_{w}}^{2}
\right]\\
&\le \sum_{w \in \mathbb{Z}^{n} \setminus \{0\}}  \lambda_{w}^{-s}
\end{align*}
which converges since $s > n/2$. Moreover, $0\le\norm{\mathsf{q}_{L}\varphi_{w}}^2_{L^{2}}-\hsp{\varphi_{w}}{\mathsf{q}_{L}\varphi_{w}}^{2}\le 1$ for all $w$ and $L$, and $\norm{\mathsf{q}_{L}\varphi_{w}}^2_{L^{2}}-\hsp{\varphi_{w}}{\mathsf{q}_{L}\varphi_{w}}^{2}\to 0$ for all $w$ as $L\to\infty$. Thus also the second   term on the RHS of \eqref{17} vanishes as $L \to \infty$.
\end{proof}


\subsubsection{Enhanced
Projection of $h$}\label{ss:Projections:D} 

Given a polyharmonic Gaussian field $h$ on the continuous torus, we define its  \emph{enhanced piecewise constant projection} or briefly  \emph{enhanced  projection} 
by 
\begin{equation}\label{h-enhanced-def}
 h_{+,L}(x)\coloneqq 
 \langle h| p_{+,L}(x,\emparg)\rangle, \qquad 
 p_{+,L}\coloneqq  q_{L}\circ r_L^{{+}}
\end{equation}
where as before
$r^{{+}}_L\coloneqq \sum_{z\in\Z^n_L} \, \frac1{\vartheta_{L,z}}\cdot\left(\frac{\lambda_z}{\lambda_{L,z}}\right)^{n/4}\cdot \varphi_z\otimes\varphi_z$ and 
$q_{L}\coloneqq L^n\sum_{v\in\mathbb{T}^n_L}{\bf 1}_{v+Q_L}\otimes {\bf 1}_{v+Q_L}$.
It is  a centered Gaussian field with  covariance function 
\begin{equation}k_{+,L}(x,y)\coloneqq
\frac1{a_n} \sum_{z\in\Z^n_L} \, \frac1{\vartheta_{L,z}^2}\cdot\frac1{\lambda_{L,z}^{n/2}}\cdot {\sf q}_L\varphi_z(x)\cdot{\sf q}_L\varphi_z(y).
\end{equation}
More precisely, it can be regarded as a piecewise constant random field on the continuous torus, or equivalently as a random field on the discrete torus.

\begin{lemma}\label{lemma-enhanced} (i) For all $f\in L^2(\mathbb{T}^n)$,
$${\sf p}_{+,L}f\to f\quad\text{in $L^2$ as $L\to\infty$}.$$

(ii) On $\mathbb{T}^n\times\mathbb{T}^n$,
$$k_{+,L}\to k\quad\text{in $L^0$ as }L\to\infty.$$
\end{lemma}
\begin{proof} (i) Recall that  ${\sf p}_{+,L}={\sf q}_L\circ {\sf r}_L^{{+}}$. From (the proof of) Example 3.12  in \cite{DHKS} we know that $\| {\sf q}_Lf-f\|_{L^2}\to 0$ as $L\to\infty$ and, by Jensen's inequality,
$$\big\| {\sf q}_Lf-{\sf q}_L{\sf r}_L^{{+}}f\big\|_{L^2}\le \big\| f-{\sf r}_L^{{+}}f\big\|_{L^2}.$$
Moreover, the latter goes to 0 as $L\to\infty$ according to
\begin{align*}
\big\| f-{\sf r}_L^{{+}}f\big\|_{L^2}^2&=\left\|
\sum_{\|z\|_\infty<L/2} \left[1- \frac1{\vartheta_{L,z}}\cdot\left(\frac{\lambda_z}{\lambda_{L,z}}\right)^{n/4}\right]\cdot\langle f,\varphi_z\rangle\,\varphi_z+ \sum_{\|z\|_\infty>L/2}\langle f,\varphi_z\rangle\,\varphi_z
\right\|^2_{L^2}\\
&=
\sum_{\|z\|_\infty<L/2} \left[1- \frac1{\vartheta_{L,z}}\cdot\left(\frac{\lambda_z}{\lambda_{L,z}}\right)^{n/4}\right]^2\cdot\langle f,\varphi_z\rangle^2
+\sum_{\|z\|_\infty>L/2}\langle f,\varphi_z\rangle^2\ \to \  0.
\end{align*}
The convergence of the last term here follows from the finiteness of $\sum_{z}\langle f,\varphi_z\rangle^2=\|f\|^2_{L^2}$.
The convergence of the first term in the last displayed formula follows from the facts that
$\left|1- \frac1{\vartheta_{L,z}}\cdot\left(\frac{\lambda_z}{\lambda_{L,z}}\right)^{n/4}\right|\le C\coloneqq (\frac\pi2)^{3n/2}$ for all $L$ and $z$, that $\left|1- \frac1{\vartheta_{L,z}}\cdot\left(\frac{\lambda_z}{\lambda_{L,z}}\right)^{n/4}\right|\to0$ for all $z$ as $L\to\infty$, and that $\sum_{z}\langle f,\varphi_z\rangle^2<\infty$.

(ii) Denote by ${\sf q}_L^{\otimes 2}$ the twofold action of ${\sf q}_L$ on functions of two variables, and put 
\begin{equation}\label{kL+}
k_L^+(x,y)\coloneqq \frac1{a_n} \sum_{z\in\Z^n_L} \, \frac1{\vartheta_{L,z}^2}\cdot\frac1{\lambda_{L,z}^{n/2}}\cdot \varphi_z(x)\cdot\varphi_z(y).
\end{equation}
Then $k_{+,L}={\sf q}_L^{\otimes 2}k_L^+$.
The claimed convergence will follow from the three subsequent convergence assertions:
\begin{enumerate}
\item[(1)] ${\sf q}_L^{\otimes 2}k(x,y)\to k(x,y)$ locally uniformly for all $x\not=y$
\item[(2)] $\iint\left| {\sf q}_L^{\otimes 2}k_L^+-{\sf q}_L^{\otimes 2}k\right|^2dxdy\le \iint\left| k_L^+-k\right|^2dxdy$
\item[(3)] $\iint\left| k_L^+-k\right|^2dxdy\to0$.
\end{enumerate}
Assertion (1) here is trivial since $k$ is smooth outside the diagonal and since $q_L$ acts with bounded support $\le 1/L\to0$. Assertion (2) follows from a simple application of Jensen's inequality.
To see assertion (3), observe that
\begin{align*}\iint\left| k_L^+-k\right|^2dxdy&=\frac1{a_n}\iint\left|
 \sum_{z\in\Z^n\setminus\{0\}} \left( \frac1{\vartheta_{L,z}^2\,\lambda_{L,z}^{n/2}}\,{\bf 1}_{\{\|z\|_\infty<L/2\}}-\frac1{\lambda_z^{n/2}}\right)\,\varphi_z(x)\,\varphi_z(y)\right|^2dx\,dy\\
 &=\frac1{a_n}\sum_{z\in\Z^n\setminus\{0\}}  \left( \frac1{\vartheta_{L,z}^2\,\lambda_{L,z}^{n/2}}\,{\bf 1}_{\{\|z\|_\infty<L/2\}}-\frac1{\lambda_z^{n/2}}\right)^2\\
  &=\frac1{a_n}\sum_{z\in\Z^n\setminus\{0\}} \frac1{(2\pi |z|)^{2n}} \left( \frac{\lambda_z^{n/2}}{\vartheta_{L,z}^2\,\lambda_{L,z}^{n/2}}\,{\bf 1}_{\{\|z\|_\infty<L/2\}}-1\right)^2.
\end{align*}
The latter converges to 0 as $L\to\infty$ since
the term $\left( \frac{\lambda_z^{n/2}}{\vartheta_{L,z}^2\,\lambda_{L,z}^{n/2}}\,{\bf 1}_{\{\|z\|_\infty<L/2\}}-1\right)
$ is bounded uniformly in $L$ and $z$, since it converges to 0 for every $z$ as $L\to\infty$, and since the sum  $\sum_{z\in\Z^n\setminus\{0\}} \frac1{(2\pi |z|)^{2n}}$ is finite.
\end{proof}

\subsubsection{Natural
Projection of $h$}\label{ss:Projections:E} 

Given a polyharmonic Gaussian field $h$ on the continuous torus, we define its    \emph{natural  projection} as the piecewise constant projection of its  Fourier projection:
\begin{equation}\label{h-natural-def}
h_{\circ,L}\coloneqq {\sf q}_L({\sf r}_L(h)).
\end{equation}
In other words, $ h_{\circ,L}(x)\coloneqq 
 \langle h| p_{\circ,L}(x,\emparg)\rangle$ where
 $p_{\circ,L}\coloneqq  q_{L}\circ r_L$
 and 
 as before
$r_L\coloneqq \sum_{z\in\Z^n_L} \,\varphi_z\otimes\varphi_z$ and 
$q_{L}\coloneqq L^n\sum_{v\in\mathbb{T}^n_L}{\bf 1}_{v+Q_L}\otimes {\bf 1}_{v+Q_L}$.
It is  a centered Gaussian field with  covariance function 
\begin{equation}k_{\circ,L}(x,y)\coloneqq
\frac1{a_n} \sum_{z\in\Z^n_L} \, \frac1{\lambda_{z}^{n/2}}\cdot {\sf q}_L\varphi_z(x)\cdot{\sf q}_L\varphi_z(y).
\end{equation}
More precisely, it can be regarded as a piecewise constant random field on the continuous torus and equivalently as a random field on the discrete torus.

\begin{lemma}\label{lemma-natural} (i) For all $f\in L^2(\mathbb{T}^n)$,
$${\sf p}_{\circ,L}f\to f\quad\text{in $L^2$ as $L\to\infty$}.$$

(ii) On $\mathbb{T}^n\times\mathbb{T}^n$,
$$k_{\circ,L}\to k\quad\text{in $L^0$ as }L\to\infty.$$
\end{lemma}

\begin{proof} Analogously to (but simpler than)  Lemma \ref{lemma-enhanced}.
\end{proof}
\subsection{Convergence Properties of the Polyharmonic Gaussian Field on the Discrete Torus and of its Extensions}

\subsubsection{Polyharmonic Gaussian Field  $h_L$ on the Discrete Torus}

\begin{theorem}\label{thm-hL} For all $f\in \bigcup\limits_{s>n/2} H^{s}(\mathbb{T}^n)$,
\[
\langle  h_L,f\rangle_{\mathbb{T}_L^n} \to \langle h,f\rangle_{\mathbb{T}^n} \quad \text{in $L^2(\mathbf P)$ as }L\to\infty \fstop
\]
\end{theorem}
\begin{proof}  Given $f=\sum_{z\in\Z^n}\alpha_z\varphi_z \in
\bigcup\limits_{s>n/2} H^{s}(\mathbb{T}^n)$, according to Lemma \ref{high-frequ},
\begin{align*}
 \langle h,f\rangle_{\mathbb{T}^n}-\langle  h_L,f\rangle_{\mathbb{T}_L^n} =&\ \frac1{\sqrt{a_n}}
 \sum_{z\in\Z_L^n\setminus\{0\}}\xi_z\cdot\left[\frac1{\lambda_{z}^{n/4}}\alpha_z-\frac1{\lambda_{L,z}^{n/4}}\sum_{w\in\Z^n}\alpha_{z+Lw}
 \right]
 \\
&+\frac1{\sqrt{a_n}}\sum_{z\in\Z^n, \, \|z\|_\infty\ge L/2}\xi_z\cdot\frac1{\lambda_{z}^{n/4}}\alpha_z\fstop
\end{align*}
Thus
\begin{align*}
\lefteqn{a_n\cdot  \mathbf E\Big[\big| \langle h,f\rangle_{\mathbb{T}^n}-\langle  h_L,f\rangle_{\mathbb{T}_L^n}\big|^2\Big]}\\
=&
 \sum_{z\in\Z_L^n\setminus\{0\}}\left[\frac1{\lambda_{z}^{n/4}}\alpha_z-\frac1{\lambda_{L,z}^{n/4}}\sum_{w\in\Z^n }\alpha_{z+Lw}
 \right]^2
 +\sum_{z\in\Z^n, \, \|z\|_\infty\ge L/2}\frac1{\lambda_{z}^{n/2}}\alpha_z^2
 \\
 \leq&\ 2 \sum_{z\in\Z_L^n\setminus\{0\}}\left[\frac1{\lambda_{z}^{n/4}}-\frac1{\lambda_{L,z}^{n/4}}
 \right]^2\alpha_z^2
 \\
 &+ 2 \sum_{z\in\Z_L^n\setminus\{0\}}\frac1{\lambda_{L,z}^{n/2}}\left[\sum_{w\in\Z^n \setminus\{0\}}\alpha_{z+Lw}
 \right]^2
 +\sum_{z\in\Z^n, \, \|z\|_\infty\ge L/2}\frac1{\lambda_{z}^{n/2}}\alpha_z^2
 \\
 \leq&\ \frac2{(2\pi)^{n}} \sum_{z\in\Z_L^n\setminus\{0\}}\left[1-\frac{\lambda_{z}^{n/4}}{\lambda_{L,z}^{n/4}}
 \right]^2\frac1{|z|^n}\alpha_z^2
 \\
 &+ \frac{2}{4^{n}}\sum_{z\in\Z_L^n\setminus\{0\}}\frac1{|z|^n}\left[\sum_{w\in\Z^n \setminus\{0\}}\alpha_{z+Lw}
 \right]^2
 +\frac1{(2\pi)^{n}}\sum_{z\in\Z^n, \, \|z\|_\infty\ge L/2}\frac1{|z|^n}\alpha_z^2\fstop
\end{align*}
Now as $L\to\infty$, the last term vanishes (since $f$ in particular lies in $H^{-n/2}(\mathbb{T}^n)$) and also the first term vanishes,
see the proof of Theorem \ref{thm3} below. 
To estimate the second term, choose $s>n/2$ with $\sum_{z\in\Z^n\setminus\{0\}}|z|^{2s}|\alpha_z|^2<\infty$, which exists by definition of~$f$. 
%
Then,
\begin{align*}
\lefteqn{\sum_{z\in\Z_L^n\setminus\{0\}}\frac1{|z|^n}\left[\sum_{w\in\Z^n \setminus\{0\}}\alpha_{z+Lw}
 \right]^2
\le\sum_{z\in\Z_L^n\setminus\{0\}}\left[\sum_{w\in\Z^n \setminus\{0\}}\alpha_{z+Lw}
 \right]^2}\\
&\le\sum_{z\in\Z_L^n\setminus\{0\}}\left[\sum_{w\in\Z^n \setminus\{0\}}\frac1{|z+Lw|^{2s}}
 \right]\cdot
 \left[\sum_{w\in\Z^n \setminus\{0\}}|z+Lw|^{2s}\cdot|\alpha_{z+Lw}|^2
 \right] \fstop
 \end{align*}
Estimating the term in the first bracket by
\[
\sum_{w\in\Z^n \setminus\{0\}}\frac1{|z+Lw|^{2s}}\leq \sum_{u\in\Z^n\setminus\{0\}}\frac1{|u|^{2s}}<\infty\comma
\]
we then obtain that
\begin{align*}
\sum_{z\in\Z_L^n\setminus\{0\}}\frac1{|z|^n}\left[\sum_{w\in\Z^n \setminus\{0\}}\alpha_{z+Lw}
 \right]^2
\leq& \left[\sum_{u\in\Z^n\setminus\{0\}}\frac1{|u|^{2s}}
 \right]\cdot
\sum_{z\in\Z_L^n\setminus\{0\}}  \left[\sum_{w\in\Z^n \setminus\{0\}}|z+Lw|^{2s}\cdot|\alpha_{z+Lw}|^2
 \right]
\end{align*}
and it therefore suffices to show that the second factor vanishes as~$L\to\infty$. Since~$z,w\neq 0$, we have that~$\abs{z+Lw}\geq L/2$, thus, relabelling~$v\coloneqq z+Lw$,
\begin{align*}
\sum_{z\in\Z_L^n\setminus\{0\}}  \left[\sum_{w\in\Z^n \setminus\{0\}}|z+Lw|^{2s}\cdot|\alpha_{z+Lw}|^2\right] =& \left[\sum_{v\in\Z^n, \, \|v\|_\infty\ge L/2}|v|^{2s}\cdot|\alpha_{v}|^2
 \right]\to 0\qquad\text{as }L\to\infty
\end{align*}
being the remainder of a convergent series.
\end{proof}


\subsubsection{Piecewise Constant Extension of $h_L$}\label{pwc-ext}
Given a polyharmonic Gaussian field $h_L$ on the discrete torus, recall that its piecewise constant extension to the continuous torus
by $h_{L,\flat}(x)\coloneqq h_L(v)$ if $x\in v+Q_L$ where $Q_L=[-\frac1{2L},\frac1{2L})^n$. Then 
\[
\langle h_{L,\flat},f\rangle_{\mathbb{T}^n} = \langle h_L,{\sf q}_Lf\rangle_{\mathbb{T}_L^n} 
\]
with ${\sf q}_Lf\in L^2(\mathbb{T}^n_L)$  for  $f\in L^2(\mathbb{T}^n)$. 
Note that ${\sf q}_Lf(v)= L^{n}\int_{v+Q_L} f(y)dy$ for $v\in\mathbb{T}^n_L$.

\begin{theorem}\label{thm-flat-L} For all $f\in \bigcup\limits_{s>n/2} H^{s}(\mathbb{T}^n)$,
\begin{equation}\label{eq:t:flat-L:0}
\langle h_{L,\flat},f\rangle_{\mathbb{T}^n} \to \langle h,f\rangle_{\mathbb{T}^n} \quad \text{in $L^2(\mathbf P)$ as }L\to\infty \fstop
\end{equation}
Furthermore, for every~$s>n/2$,
\begin{equation}
\big\|h_{L,\flat}-h\big\|_{\mathring{H}^{-s}}^2\to 0 \quad \text{in $L^2(\mathbf P)$ as }L\to\infty\ .
\end{equation}
\end{theorem}
\begin{proof}  
By construction,
\[
\langle h_{L,\flat},f\rangle_{\mathbb{T}^n} - \langle h,f\rangle_{\mathbb{T}^n} = \langle h_L,{\sf q}_Lf-f\rangle_{\mathbb{T}_L^n}+
 \langle h_L,f\rangle_{\mathbb{T}_L^n}
 - \langle h,f\rangle_{\mathbb{T}^n} 
\]
and according to the previous Theorem \ref{thm-hL}, $ \langle h_L,f\rangle_{\mathbb{T}_L^n}
 - \langle h,f\rangle_{\mathbb{T}^n} \to0$ as $L\to\infty$.
 The first claim thus follows from
 \begin{align*}
\mathbf E\Big[\langle h_L,{\sf q}_Lf-f\rangle_{\mathbb{T}_L^n}^2\Big]&=\frac1{a_n}\sum_{z\in\Z_L^n\setminus\{0\}}
\frac1{\lambda_{L,z}^{n/2}}\langle \varphi_z, {\sf q}_Lf-f\rangle_{\mathbb{T}_L^n}^2\\
&\le\frac1{4^n\, a_n}\sum_{z\in\Z_L^n\setminus\{0\}}
\frac1{|z|^n}\langle \varphi_z, {\sf q}_Lf-f\rangle_{\mathbb{T}_L^n}^2\\
&=
\frac1{4^n\, a_n}\,\big\|{\sf q}_Lf-f\big\|^2_{H^{-n/2}}\\
&\le \frac1{4^n\, a_n}\,\big\|{\sf q}_Lf-f\big\|^2_{L^2}\to0\quad\text{as $L\to\infty$}
\end{align*}
since by Sobolev embedding
\[
\bigcup\limits_{s>n/2} H^{s}(\mathbb{T}^n)\subset \mathcal C(\mathbb{T}^n) \fstop
\]

%
For the second claim, observe that
\begin{equation*}
  h_{L,\flat} = \frac{1}{\sqrt{a_{n}}} \sum_{z \in \mathbb{Z}^{n} \setminus \{0\}} \varphi_{z} \sum_{w \in \mathbb{Z}^{n}_{L} \setminus \{0\}} \lambda_{w,L}^{-n/4} \xi_{w} \hsp{\varphi_{w}}{\mathsf{q}_{L}\varphi_{z}}.
\end{equation*}
From this, we get
\begin{equation*}
  a_{n} \norm{h_{L,\flat}}^{2}_{H^{-s}} = \sum_{z \in \mathbb{Z}^{n} \setminus \{0\}} \lambda_{z}^{-s} \paren{ \sum_{w \in \mathbb{Z}^{n}_{L} \setminus \{0\}} \lambda_{w,L}^{-n/4} \xi_{w} \hsp{\varphi_{w}}{\mathsf{q}_{L}\varphi_{z}} }^{2}.
\end{equation*}
And thus
\begin{equation*}
  a_{n} \mathbf{E} \norm{h_{L,\flat}}^{2}_{H^{-s}} = \sum_{z \in \mathbb{Z}^{n} \setminus \{0\}} \lambda_{z}^{-s} \sum_{w \in \mathbb{Z}^{n}_{L} \setminus \{0\}} \lambda_{w,L}^{-n/2} \hsp{\varphi_{w}}{\mathsf{q}_{L}\varphi_{z}}^{2} \\
\end{equation*}
On the one hand, as $L \to \infty$, the summand converges to $\lambda_{z}^{-s-n/2} 1_{z = w}$.
On the other hand, by Parseval's indentity, we find that
\begin{equation*}
  \sum_{w \in \mathbb{Z}^{n}_{L} \setminus \{0\}} \lambda_{w,L}^{-n/2} \hsp{\varphi_{w}}{\mathsf{q}_{L}\varphi_{z}}^{2} = \norm{{(-\Delta_{L})}^{-n/2} \mathsf{q}_{L}\varphi_{z}}_{L^{2}(\mathbb{T}^{n}_{L})}^{2} \leq 1.
\end{equation*}
Thus the sum is uniformly bounded since $s > n/2$.
This shows that
\begin{equation*}
  \mathbf{E} \norm{h_{L,\flat}}^{2}_{H^{-s}} \to \mathbf{E} \norm{h}^{2}_{H^{-s}}.
\end{equation*}
A similar computation shows that
\begin{equation*}
  \mathbf{E} \hsp{h}{h_{L,\flat}}_{H^{-s}} \to \mathbb{E} \norm{h}_{H^{-s}}^{2}.
\end{equation*}
This concludes the proof of the second claim.
\end{proof}

%
%
%
\subsubsection{Fourier Extension of $h_L$}\label{sss:FourierExtension}

Given a polyharmonic Gaussian field $h_L$ on the discrete torus, we define its Fourier extension to the continuous torus as 
$$h_{L,\sharp}(x)\coloneqq \dot{\sf r}_L h(x)\coloneqq \langle h, r_L(x,\emparg)\rangle_{\mathbb{T}^N_L} \qquad (\forall x\in\mathbb{T}^n)\fstop$$
In other words, for every $\omega$ the function $h^\omega_{L,\sharp}$ is the unique function in ${\mathcal D}_L$ with $h^\omega_{L,\sharp}=h^\omega_L$ on $\mathbb{T}^n_L$, cf.~\eqref{haL}.
If $h_L$ is given as 
\begin{equation*}
 h_L(v)\coloneqq \frac1{\sqrt{a_n}}\sum_{z\in\Z_L^n\setminus\{0\}} 
\frac1{\lambda_{L,z}^{n/4}}
\cdot 
\xi_z\cdot \varphi_z(v) \qquad (\forall v\in\mathbb{T}_L^n)
\end{equation*}
then the same formula provides the representation of $h_{L,\sharp}$:
\begin{equation}\label{haLtil2}
 h_{L,\sharp}(x)\coloneqq \frac1{\sqrt{a_n}}\sum_{z\in\Z_L^n\setminus\{0\}} 
\frac1{\lambda_{L,z}^{n/4}}
\cdot 
\xi_z\cdot \varphi_z(x) \qquad (\forall x\in\mathbb{T}^n)\fstop
\end{equation}
This is a centered Gaussian field on $\mathbb{T}^n$ with covariance function
\begin{equation}\label{kL-sharp}
k_{L,\sharp}(x,y)=\frac1{a_n}\sum_{z\in\Z^n_L\setminus\{0\}} \frac1{\lambda_{L,z}^{n/2}}\cdot 
\cos\Big(2\pi \, z\cdot (x-y)\Big) \qquad(\forall x,y\in\mathbb{T}^n)
\end{equation}
 
%

\begin{figure}[htb!]
\includegraphics[scale=.5]{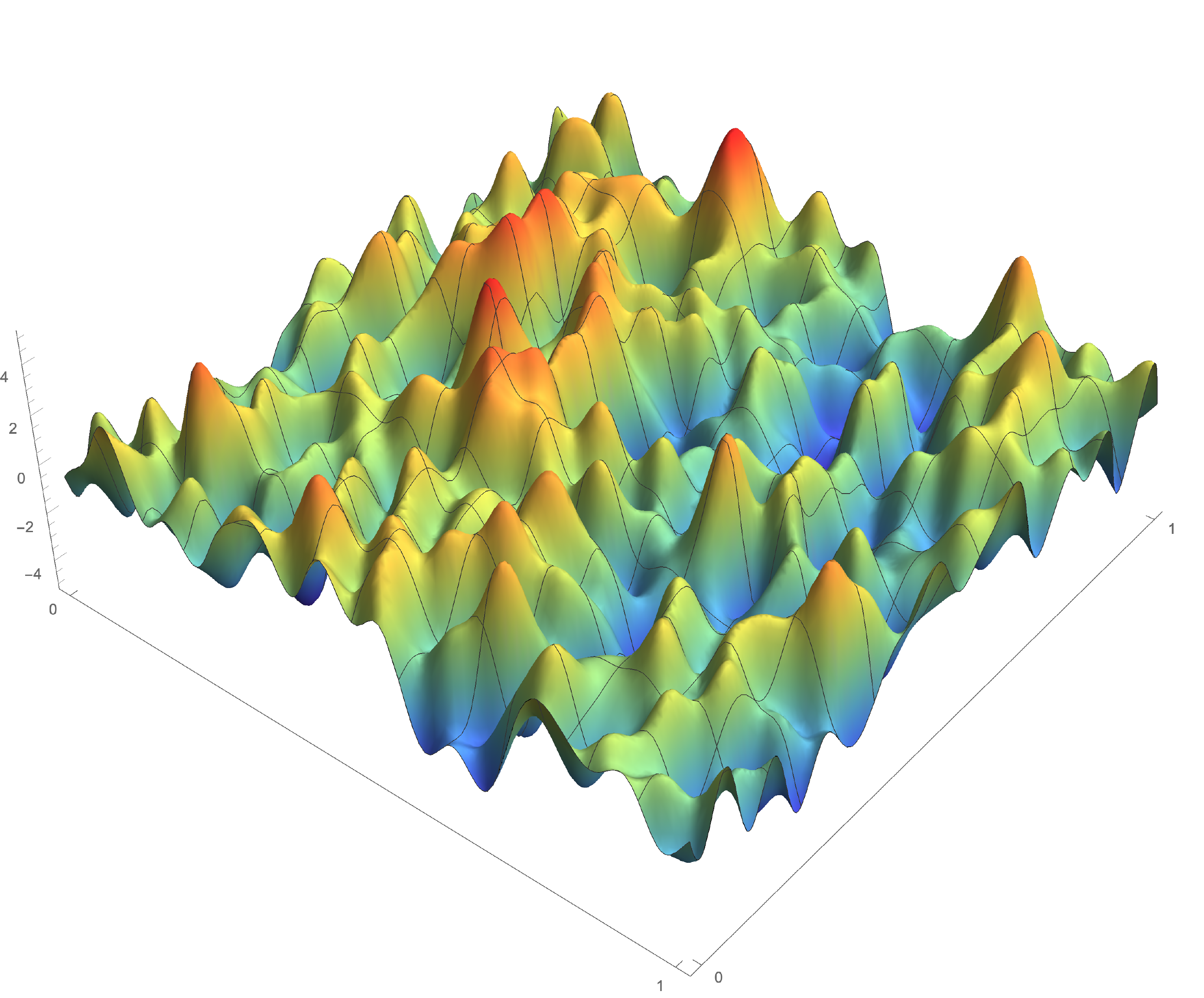}
\caption{A realization of~$h_{25,\sharp}$}
\end{figure}

\begin{theorem}\label{thm3} For all $f\in \mathring H^{-n/2}(\mathbb{T}^n)$,
\begin{equation}
\langle h_{L,\sharp},f\rangle_{\mathbb{T}^n} \to \langle h,f\rangle_{\mathbb{T}^n} \quad \text{in $L^2(\mathbf P)$ as }L\to\infty \comma
\end{equation}
and for all $\epsilon>0$,
\begin{equation}\big\|h_{L,\sharp}-h\big\|_{H^{-\epsilon}}^2\to 0 \quad \text{in $L^2(\mathbf P)$ as }L\to\infty\fstop
\end{equation}
\end{theorem}

\begin{proof} According to the Proposition \ref{thm-four-L}, we already know that $\langle h_{\sharp,L},f\rangle \to \langle h,f\rangle$. Thus it suffices to prove
 $\langle h_{L,\sharp}-h_{\sharp,L},f\rangle \to 0$.
This follows according to
\begin{align*}
\lefteqn{\mathbf E\Big[\langle h_{L,\sharp}-h_{\sharp,L},f\rangle^2\Big]}\\
&=
\frac1{a_n}\sum_{z\in\Z_L^n\setminus\{0\}} \Bigg(\frac1{\lambda_{L,z}^{n/2}}
-\frac1{\lambda_z^{n/2}}\Bigg)\, \langle\varphi_z,f\rangle^2\\
&=
\frac1{a_n}\sum_{z\in\Z_L^n\setminus\{0\}} \Bigg(\bigg(\frac{\lambda_z}{\lambda_{L,z}}\bigg)^{n/2}-1\Bigg)\, \frac{\langle\varphi_z,f\rangle^2}{\lambda_z^{n/2}}\to 0
\end{align*}
as~$L\to\infty$ by the Dominated Convergence Theorem since 
\begin{gather*}
\sum_{z\in\Z^n\setminus\{0\}}  \frac{1}{\lambda_z^{n/2}}\langle\varphi_z,f\rangle^2<\infty\comma
\\
0\le\bigg(\frac{\lambda_z}{\lambda_{L,z}}\bigg)^{n/2}-1\le 2^n
\end{gather*}
for all $z$ and $L$, and
\[
\bigg(\frac{\lambda_z}{\lambda_{L,z}}\bigg)^{n/2}-1 \to 0
\]
as~$L\to\infty$ for every $z$.

\medskip

For the second claim, we use the fact that  again by Proposition \ref{thm-four-L} for every $\epsilon>0$,
$$\mathbf E\Big[\big\| h-h_{\sharp, L}\big\|_{\mathring H^{-\epsilon}}^2
\Big]\to0\quad\text{as }L\to\infty\fstop$$
Furthermore,
\begin{align*}
\mathbf E\Big[\big\| h_{L,\sharp}-h_{\sharp, L}\big\|_{\mathring H^{-\epsilon}}^2
\Big]
&=
\frac1{a_n}\sum_{z\in\Z_L^n\setminus\{0\}} \Bigg(\frac1{\lambda_{L,z}^{n/4}}
-\frac1{{(2\pi|z|)}^{n/2}}\Bigg)^2\, \frac1{|z|^{2\epsilon}}\\
&=
\frac1{a_n(2\pi)^n}\sum_{z\in\Z_L^n\setminus\{0\}} \Bigg(\bigg(\frac{\lambda_z}{\lambda_{L,z}}\bigg)^{n/4}-1\Bigg)^2\, \frac{1}{{|z|}^{n-2\epsilon}}\to0
\end{align*}
as $L\to\infty$ by the same Dominated Convergence arguments as for the first claim.
\end{proof}

\begin{figure}[htb!]
     \centering
     \begin{subfigure}[b]{0.45\textwidth}
         \centering
         \includegraphics[width=\textwidth]{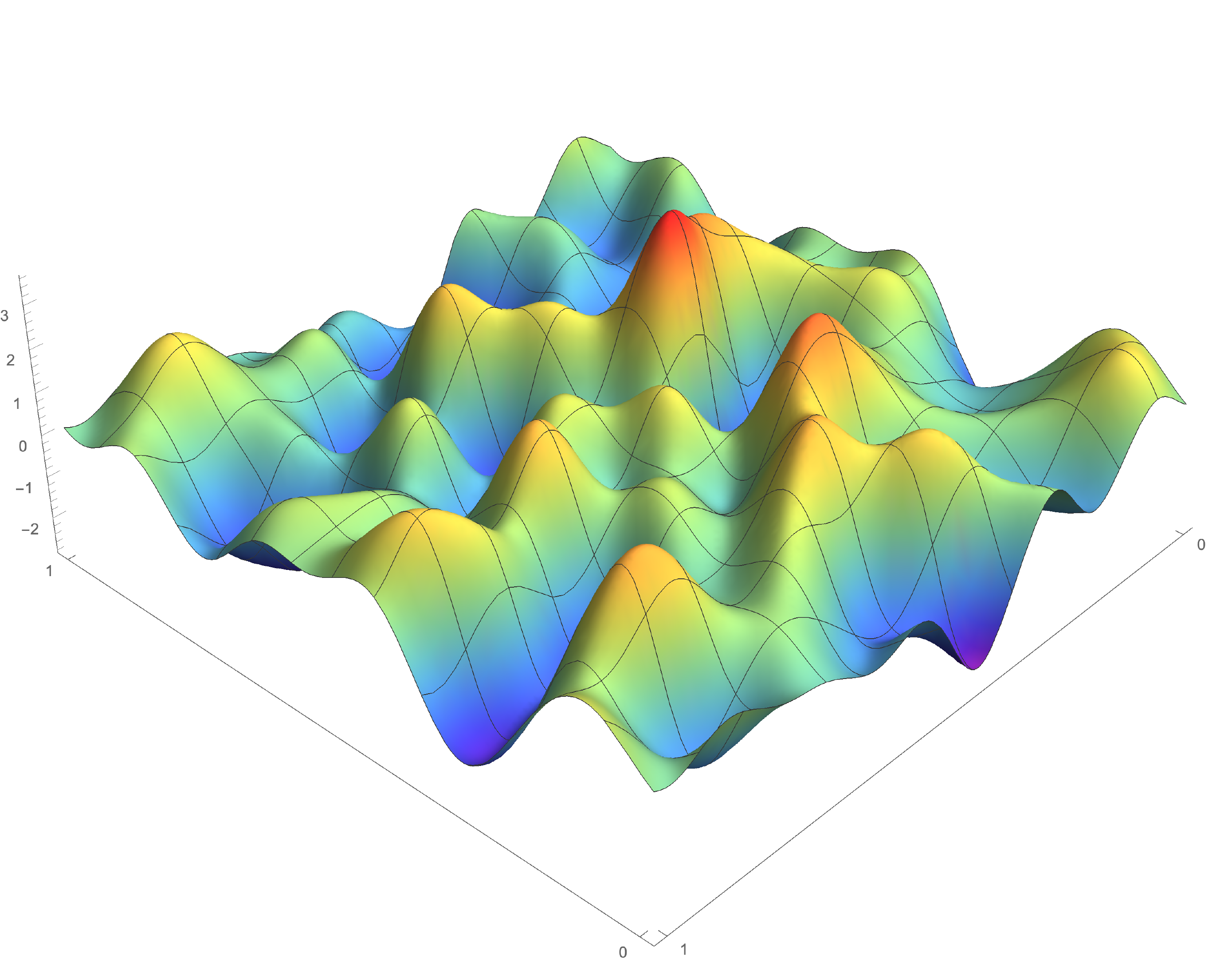}
         \caption{$h_{11,\sharp}$}
     \end{subfigure}
     \hfill
     \begin{subfigure}[b]{0.45\textwidth}
         \centering
         \includegraphics[width=\textwidth]{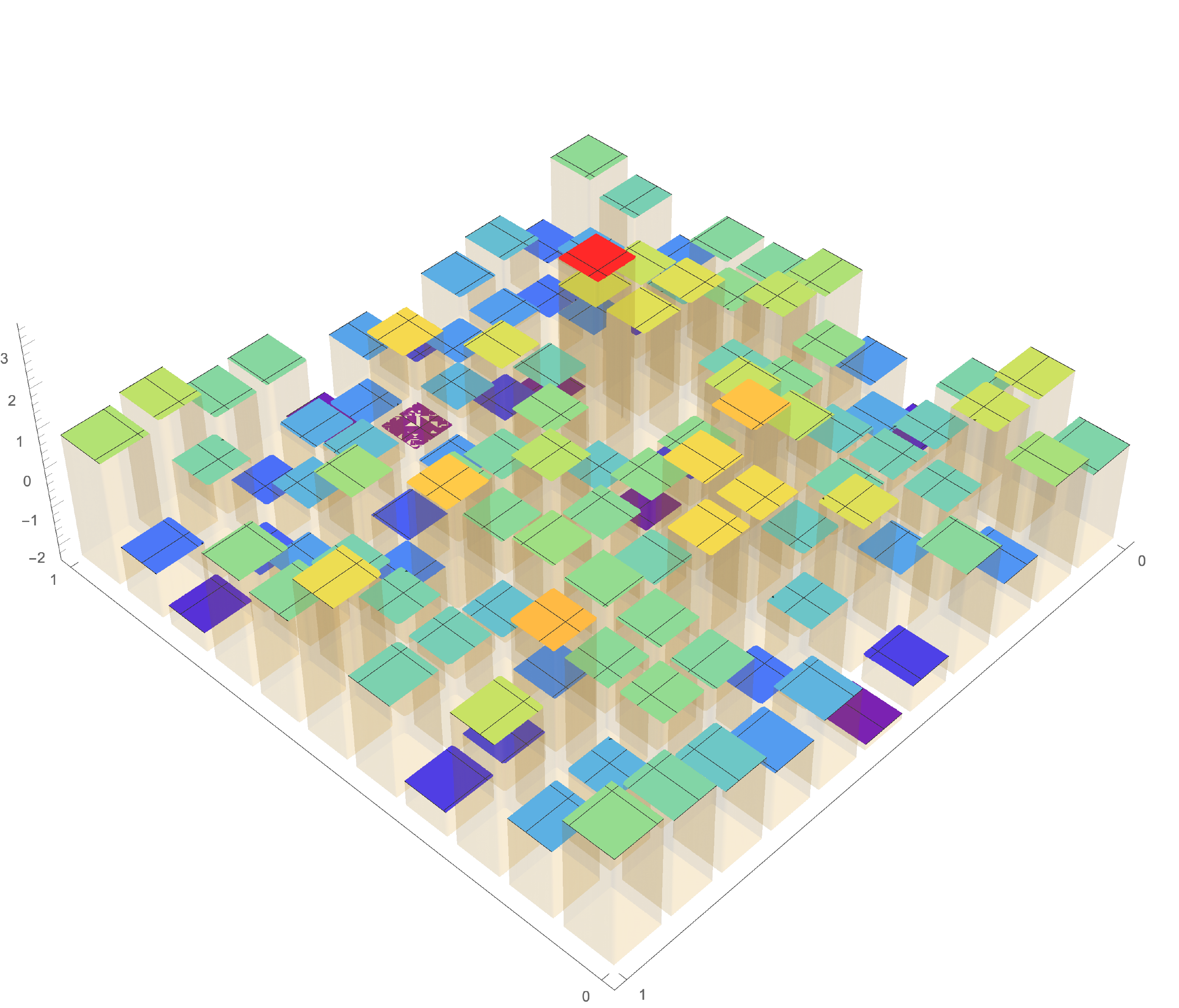}
         \caption{$h_{11,\flat}$}
     \end{subfigure}
        \caption{\footnotesize Two different approximations/extensions of the field~$h$ with same realization of the randomness.}
\end{figure}

\subsubsection{Reduced and Spectrally Reduced Polyharmonic Gaussian Fields  on the Discrete Torus}\label{sss:ReducedDiscrete}

The previous assertions Theorem \ref{thm-hL}, \ref{thm-flat-L}, and \ref{thm3}  on the polyharmonic field  and its (Fourier or piecewise constant, resp.) extensions --- as well as their proofs --- hold true in analogous form 
for the reduced and spectrally reduced polyharmonic fields $h_L^{{-}}$ and $h_L^{{-\circ}}$ and their (Fourier or piecewise constant, resp.) extensions $h_{L,\sharp}^{{-}}$, $h_{L,\flat}^{{-}}$ and $h_{L,\sharp}^{{-\circ}}$, $h_{L,\flat}^{{-\circ}}$.

\begin{proposition}\label{thm-mod-hL} (i) For all $f\in \bigcup\limits_{s>n/2} H^{s}(\mathbb{T}^n)$,
\[
\langle h^{{-\circ}}_L,f\rangle_{\mathbb{T}_L^n} \to \langle  h,f\rangle_{\mathbb{T}^n} \quad \text{in $L^2(\mathbf P)$ as }L\to\infty 
\]
as well as
\[
\langle h^{{-\circ}}_{L,\flat},f\rangle_{\mathbb{T}^n} \to \langle h,f\rangle_{\mathbb{T}^n} \quad \text{in $L^2(\mathbf P)$ as }L\to\infty \fstop
\]
(ii) 
For all $f\in \mathring H^{-n/2}(\mathbb{T}^n)$,
\[
\langle h^{{-\circ}}_{L,\sharp},f\rangle_{\mathbb{T}^n} \to \langle h,f\rangle_{\mathbb{T}^n} \quad \text{in $L^2(\mathbf P)$ as }L\to\infty \fstop
\]
(iii) Analogously with $h_{L}^{{-}}$, $h_{L,\flat}^{{-}}$, $h_{L,\sharp}^{{-}}$ in the place of $h_{L}^{{-\circ}}$, $h_{L,\flat}^{{-\circ}}$, $h_{L,\sharp}^{{-\circ}}$.
\end{proposition}


\subsubsection{Related Approaches and  Results}\label{remarks}
\begin{remark}[Log-correlated Gaussian fields in the continuum]
For~$n=1$, the field~$h$ is the lower limiting case of the fractional Brownian motion with (regularity) parameter in~$(\tfrac{1}{2},\tfrac{3}{2})$, see e.g.~\cite{LodSheSunWat16,DuplantierRhodesSheffieldVargas}.
For~$n=2$, it is the celebrated \emph{Gaussian Free Field} (GFF) on~$\mathbb{T}^2$, surveyed in~\cite{Sheffield}.
For~$n=1$, it coincides in distribution with the restriction of a GFF to a line ($\mathbb{T}^1\subset \mathbb{T}^2$).
For~$n\geq 3$, it is a log-correlated Gaussian field surveyed in~\cite{DuplantierRhodesSheffieldVargas}.
The conformal invariance of~$h$ on~$\mathbb{T}^n$ is a consequence of the conformal invariance of the Laplace--Beltrami operator on flat geometries.
The correct (i.e., conformally covariant) construction of log-correlated Gaussian fields in general non-flat geometries may be found in~\cite{DHKS}.
\end{remark}

\begin{remark}[Discrete-to-continuum approximation]
To the best of our knowledge, no discretization/discrete approximation results are available for log-correlated Gaussian fields in dimension~$n\geq 5$. In small dimension however, Gaussian fields analogous to~$h$ are known to be scaling limits of different discrete models.
For $n=2$, in light of the celebrated universality property of GFFs, the field~$h$ is a scaling limit for a huge number of different discrete Gaussian and non-Gaussian fields defined in various settings, from lattices to random environments, as, e.g., the random conductance model~\cite{Biskup}.
For~$n=4$, the field~$h$ is generated by the Neumann bi-Laplacian; the analogous field generated by the Dirichlet bi-Laplacian on~$[0,1]^4$ is the scaling limit of the membrane model~\cite{Sakagawa,Kurt}, see~\cite{CiprianiDanHazra,Schweiger}, as well as of the odometer for the divisible sandpile model~\cite{LevineMuruganPeresUgurcan}, see~\cite{CiprianiHazraRuszel}.
We stress that our convergence results for different discretizations of~$h$ hold in~$\mathring{H}^{-s}$ with~$s>2$, thus matching the same range of exponents as for the scaling limit of the sandpile odometer, see~\cite[Prop.~14]{CiprianiGraaffRuszel}.
On the other hand, the analogous scaling limit for the membrane model has so far been proven only in~$H^{-s}$ for~$s>6$, see \cite[Thm.~3.11]{CiprianiDanHazra}.
\end{remark}
\subsection{Identifications}
\begin{lemma}\label{fourier-ext/rest} The spectral enhancement $h^{{+\circ}}_{\sharp,L}$ of the Fourier projection $h_{\sharp,L}$ of the polyharmonic field~$h$ coincides in distribution with the Fourier extension $h_{L,\sharp}$ of the discrete polyharmonic field $h_L$. 

Similarly,  the Fourier approximation $ h_{\sharp,L}$ of the polyharmonic field $h$ coincides in distribution with the spectral reduction $h^{{-\circ}}_{L,\sharp}$ of the Fourier extension $h_{L,\sharp}$ of the discrete polyharmonic field $h_L$ (which in turn coincides with the Fourier extension of the spectrally reduced discrete polyharmonic field).
Or, in other terms,
the restriction of $h_{\sharp,L}$   onto $\mathbb{T}^n_L$ defines a spectrally reduced polyharmonic field~$ h_L^{{-\circ}}$. 
 \end{lemma}
 
\begin{lemma}\label{lem-enhanced-disc} Restricted to the discrete torus,  the enhanced projection $h_{+,L}$ of the polyharmonic Gaussian field $h$ on the continuous torus  coincides in distribution with the polyharmonic Gaussian field  $h_L$ on the discrete torus.
\end{lemma}

\begin{proof} 
Assume without restriction that $h$ is given in terms of standard iid normal variables $(\xi_z)_{z\in\Z^n\setminus\{0\}}$ as
$$h=\frac1{\sqrt{a_n}}\sum^{\Box}_{z\in\Z^n\setminus\{0\}}\frac1{\lambda_z^{n/4}}\,\xi_z\,\varphi_z.$$
Then
$$
{\sf r}_L^{{+}}(h)=\frac1{\sqrt{a_n}}\sum^{\Box}_{z\in\Z^n\setminus\{0\}}\frac1{\vartheta_{L,z}}\,\frac1{\lambda_{L,z}^{n/4}}\,\xi_z\,\varphi_z$$
on $\mathbb{T}^n$. Thus for $v\in\mathbb{T}^n_L$,
\begin{align*}h_{+,L}(v)&=\frac1{\sqrt{a_n}}\sum^{\Box}_{z\in\Z^n\setminus\{0\}}\frac1{\vartheta_{L,z}}\,\frac1{\lambda_{L,z}^{n/4}}\,\xi_z\cdot L^n\,\int_{v+Q_L}\varphi_z(y)dy\\
&=\frac1{\sqrt{a_n}}\sum^{\Box}_{z\in\Z^n\setminus\{0\}}\frac1{\lambda_{L,z}^{n/4}}\,\xi_z\,\varphi_z(v)
\end{align*}
according to Lemma \ref{averages}.
Therefore, by Proposition \ref{discrete-iid}, $h_{+,L}$ is distributed according to the polyharmonic Gaussian field on the discrete torus $\mathbb{T}^n_L$.
\end{proof}

\begin{lemma}\label{lem-natural-dis} Let $h$ be a polyharmonic Gaussian field on the continuous torus $\mathbb{T}^n$ and let
$$h_{\circ,L}\coloneqq {\sf q}_L({\sf r}_L(h))$$
denote the piecewise constant projection of its  Fourier projection, called \emph{natural projection of $h$}.
 Then  this field on $\mathbb{T}^n$ coincides in distribution with the piecewise constant extension of the reduced discrete polyharmonic field  $h^{{-}}_{L,\flat}$. In particular, the field   $h_{\circ,L}$ coincides 
  on $\mathbb{T}^n_L$  in distribution with the reduced discrete polyharmonic field  $h_{L}^{{-}}$.  \end{lemma}

\begin{proof} 
Assume without restriction that 
$$h=\frac1{\sqrt{a_n}}\sum^{\Box}_{z\in\Z^n\setminus\{0\}}\frac1{\lambda_z^{n/4}}\,\xi_z\,\varphi_z.$$
Then
\begin{align*}h_{\circ,L}(x)&=\frac1{\sqrt{a_n}}\sum^{\Box}_{z\in\Z^n\setminus\{0\}}\frac1{\lambda_{z}^{n/4}}\,\xi_z\cdot\sum_{v\in\mathbb{T}^n_L}{\bf 1}_{v+Q_L}(x)\cdot L^n\,\int_{v+Q_L}\varphi_z(y)dy,\end{align*}
and thus for $v\in \mathbb{T}^n_L$,
\begin{align*}h_{\circ,L}(v)&=\frac1{\sqrt{a_n}}\sum^{\Box}_{z\in\Z^n\setminus\{0\}}\frac1{\lambda_{z}^{n/4}}\,\xi_z\cdot L^n\,\int_{v+Q_L}\varphi_z(y)dy\\
&=\frac1{\sqrt{a_n}}\sum^{\Box}_{z\in\Z^n\setminus\{0\}}\frac{\vartheta_{L,z}}{\lambda_{z}^{n/4}}\,\xi_z\,\varphi_z(v)
\end{align*}
according to Lemma \ref{averages}.
Therefore,  $h_{\circ,L}$ is distributed on  $\mathbb{T}^n_L$ according to the reduced discrete polyharmonic Gaussian field.
\end{proof}

%

\section{LQG Measures on Discrete and Continuous Tori} 
We will introduce and analyze \emph{Liouville Quantum Gravity (= LQG) measures} on discrete and continuous tori. Our main result in this section will be 
that as $L\to\infty$ the  LQG measures on the discrete tori  $\mathbb{T}^n_L$ will converge to the LQG measure on the continuous torus $\mathbb{T}^n$.

An analogous convergence assertion in greater generality will be proven for the so-called \emph{reduced LQG measures}, random measures on the discrete tori $\mathbb{T}^n_L$ defined in terms of the discrete polyharmonic fields $h_L$.

\subsection{LQG Measure on the Continuous Torus and its Approximations}
For $\gamma\in\R$,
define the random measure $\mu_{\sharp,L}$ on $\mathbb{T}^n$ by
\[
d\mu_{\sharp,L}(x)=\exp\Big(\gamma h_{\sharp,L}(x)-\frac{\gamma^2}2 k_{\sharp,L}(x,x)\Big)\,d\Leb^n(x) 
\]
where $h_{\sharp,L}$ denotes the \emph{Fourier projection} (or \emph{eigenfunction approximation}) of the polyharmonic field $h$ 
and $k_{\sharp,L}$ the associated covariance function (which has constant value $\frac1{a_n}\sum_{z\in\Z^n_L\setminus\{0\}} \frac1{(2\pi|z|)^{n}}$ on the diagonal)  as introduced in \eqref{haLtil} and \eqref{kaLtil}.

\begin{proposition}[{\cite[Thms.~4.1+4.15]{DHKS}}]  \label{four-meas}
Assume $|\gamma|<\gamma^*\coloneqq\sqrt{2n}$. Then  for $\mathbf P$-a.e.~$\omega$, a unique Borel measure $\mu^\omega$ on $\mathbb{T}^n$ exists with
\[
\mu^\omega_{\sharp,L}\to \mu^\omega\qquad\text{as }L\to\infty
\]
in the sense of weak convergence of measures on $\mathbb{T}^n$ (i.e.~tested against $f\in\mathcal C(\mathbb{T}^n)$).
Even more, for every $f\in L^1(\mathbb{T}^n)$,
$$\int_{\mathbb{T}^n} f\,d\mu_{\sharp,L}\to \int_{\mathbb{T}^n}f\,d\mu\qquad\text{as }L\to\infty\quad\mathbf P\text{-a.s.~and in }L^1(\mathbf P)\fstop$$
For $|\gamma|<\sqrt n$, the latter convergence 
also holds in $L^2(\mathbf P)$.
\end{proposition}

\begin{definition} The random measure $\mu$ on $\mathbb{T}^n$ constructed and characterized above is called \emph{polyharmonic LQG measure}  or \emph{polyharmonic Gaussian multiplicative chaos}.

The random measures $\mu_{\sharp,L}$ on $\mathbb{T}^n$ are called Fourier approximations of the polyharmonic LQG measure.
\end{definition}

\paragraph{Piecewise Constant Approximation.}
For $\gamma\in\R$,
define the random measure $\mu_{\flat,L}$ on $\mathbb{T}^n$ by
\[
d\mu_{\flat,L}(x)=\exp\Big(\gamma h_{\flat.L}(x)-\frac{\gamma^2}2 k_{\flat,L}(x,x)\Big)\,d\Leb^n(x) 
\]
where $h_{\flat,L}$ denotes the \emph{piecewise constant projection} of the polyharmonic field $h$ 
and $k_{\flat,L}$ the associated covariance function (which is constant  on the diagonal)  as introduced in \eqref{h-flat-def} and \eqref{k-flat-def}.

\begin{proposition}[{\cite[Thm.~4.14]{DHKS}}]  Assume $|\gamma|<\sqrt{2n}$. Then 
in~$L^1(\mathbf P)$, 
\[
\mu_{\flat,L}\to\mu\qquad\text{as }L\to\infty.
\]
\end{proposition}

\begin{proof}
  This theorem is proven in \cite[Thm. 4.14]{DHKS} using Kahane's convexity inequality.
\end{proof}

\paragraph{Enhanced Approximation.}
For $\gamma\in\R$,
define the random measure $\mu_{+,L}$ on $\mathbb{T}^n$ by
\[
d\mu_{+,L}(x)=\exp\Big(\gamma h_{+,L}(x)-\frac{\gamma^2}2 k_{+,L}(x,x)\Big)\,d\Leb^n(x) 
\]
where $h_{+,L}$ denotes the \emph{enhanced projection} of the polyharmonic field $h$ 
and $k_{+,L}$ the associated covariance function.

\begin{proposition}\label{thm-conv-enh}  Assume $|\gamma|<\sqrt{n/e}$. Then 
in~$L^1(\mathbf P)$, 
\[
\mu_{+,L}\to\mu\qquad\text{as }L\to\infty.
\]
\end{proposition}

\begin{proof}  
To show the convergence of the LQG measures~$\mu_{+,L}$ associated with the enhanced projections of the polyharmonic field on the torus $\mathbb{T}^n$  to the LQG measure~$\mu$ on~$\mathbb{T}^n$,
 we verify the necessary assumptions in~\cite[Lem.~4.5]{DHKS}, a rewriting in the present setting of the general construction of Gaussian Multiplicative Chaoses by A.~Shamov,~\cite{Sha16}.
Lemma \ref{lemma-enhanced} provides the convergence results for the regularizing kernel $p_{+,L}$ and for the covariance kernel $k_{+,L}$.
The uniform integrability of the approximating sequence of random measures ~$\mu_{+,L}$ will be proven as Theorem \ref{p:Estimate} in the last section.

\end{proof}

\paragraph{Natural Approximation.}
For $\gamma\in\R$,
define the random measure $\mu_{\circ,L}$ on $\mathbb{T}^n$ by
\[
d\mu_{\circ,L}(x)=\exp\Big(\gamma h_{\circ,L}(x)-\frac{\gamma^2}2 k_{\circ,L}(x,x)\Big)\,d\Leb^n(x) 
\]
where $h_{\circ,L}$ denotes the \emph{natural projection} of the polyharmonic field $h$ 
and $k_{\circ,L}$ the associated covariance function.

\begin{proposition}\label{thm-conv-nat}  Assume $|\gamma|<\sqrt{n}$. Then 
in~$L^1(\mathbf P)$, 
\[
\mu_{\circ,L}\to\mu\qquad\text{as }L\to\infty.
\]
\end{proposition}

\begin{proof}  
To show the convergence of the LQG measures~$\mu_{\circ,L}$ associated with the natural projections of the polyharmonic field $h$ to the LQG measure~$\mu$ on~$\mathbb{T}^n$,
 we again verify the necessary assumptions in~\cite[Lem.~4.5]{DHKS}.
Lemma \ref{lemma-natural} provides the convergence results for the regularizing kernel $p_{\circ,L}$ and for the covariance kernel $k_{\circ,L}$.
The uniform integrability --- even $L^2$-boundedness --- of the approximating sequence of random measures~$\mu_{\circ,L}$ follows from the $L^2$-boundedness of the sequence of random measures~$\mu_{\sharp,L}$ as stated in Proposition \ref{four-meas} and a straightforward application of Jensen's inequality with the Markov kernel $q_L$:
\begin{align*}
\sup_{L}\mathbb E\left[\left| \mu_{\circ,L}(\mathbb{T}^n)\right|^2\right]&=
\sup_{L}\mathbb E\left[\left| \int_{\mathbb{T}^n} \exp\left(\gamma h_{\circ,L}(x)-\frac{\gamma^2}2k_{\circ,L}(x)\right)dx\right|^2\right]\\
&=\sup_{L}\iint_{\mathbb{T}^n\times\mathbb{T}^n}\exp\left(\gamma^2\,k_{\circ,L}(x,y)\right)dy\,dx\\
&\le\sup_{L}\iint_{\mathbb{T}^n\times\mathbb{T}^n}\iint_{\mathbb{T}^n\times\mathbb{T}^n}\exp\left(\gamma^2\,k_{\sharp,L}(x',y')\right) q_L(x,x')\,q_L(y,y')\,dy'\,dx'\,dy\,dx\\
&=\sup_{L}\iint_{\mathbb{T}^n\times\mathbb{T}^n}\exp\left(\gamma^2\,k_{\sharp,L}(x,y)\right)dy\,dx\\
&=\sup_{L}\mathbb E\left[\left| \mu_{\sharp,L}(\mathbb{T}^n)\right|^2\right]=\mathbb E\left[\left| \mu(\mathbb{T}^n)\right|^2\right]<\infty\fstop
\end{align*}
{\em 
Alternatively,} we can also use Jensen's inequality directly at the level of $h_{\circ,L}$.
Namely, we find that
\begin{equation*}
  \begin{split}
    \exp \left( \gamma h_{\circ,L}(x) - \frac{\gamma^{2}}{2} k_{\circ,L}(x,x)\right) &= \exp \left( \int \left(\gamma h_{\sharp,L}(x') - \frac{\gamma^{2}}{2} k_{\sharp,L}(x', x'') \right) q_{L}(x, x') q_{L}(x, x'') dx dx' dx'' \right) \\
                                                                                            &\leq \int \exp \left( \gamma h_{\sharp,L}(x') - \frac{\gamma^{2}}{2} k_{\sharp,L}(x', x'') \right) q_{L}(x,x') q_{L}(x,x'').
  \end{split}
\end{equation*}
From which, we get
\begin{equation*}
  \mu_{\circ,L}(\mathbb{T}^{n}) \leq \mu_{\sharp,L}(\mathbb{T}^{n}).
\end{equation*}
\end{proof}

\subsection{LQG Measures on the Discrete Tori and their Convergence}

In this section, we will prove our main result: the convergence of the  discrete LQG measures~$\mu_L$ to the LQG measure~$\mu$ on~$\mathbb{T}^n$. 

Whereas this convergence  only holds for a restricted range of parameters $\gamma$, 
the convergence of the so-called reduced discrete LQG measures~$\mu_L$ to the LQG measure~$\mu$ on~$\mathbb{T}^n$ will hold in 
greater generality.

\begin{definition}  For given $\gamma\in\R$, the \emph{polyharmonic LQG measure on the discrete torus  $\mathbb{T}_L^n$} or \emph{discrete LQG measure}
is the random measure $\mu_L$ on  $\mathbb{T}_L^n$ defined by
\[
d\mu_L(v)=\exp\Big(\gamma h_L(v)-\frac{\gamma^2}2 k_L(v,v)\Big)\,dm_L(v)
\fstop
\]
\end{definition}
Here $h_{L}$ is the 
polyharmonic Gaussian field on the discrete torus $\mathbb{T}^n_L$,
$k_{L}$ its covariance function   as introduced in \eqref{haL} and \eqref{kL}, and $m_L$ the normalized counting measure
$\frac1{L^n}\sum_{u\in\mathbb{T}^n_L}\ \delta_u$ on the discrete torus.
Recall that  
$k_L(v,v)=\frac1{a_n}\sum_{z\in\Z^n_L\setminus\{0\}} \frac1{\lambda_{L,z}^{n/2}}$
for all $v\in \mathbb{T}^n_L$.


For proving  convergence of the random measures $\mu_L$ on the discrete tori $\mathbb{T}_L^n$ as $L\to\infty$, we will restrict ourselves to subsequences for which the discrete tori are hierarchically  ordered, say $L=a^\ell$ as  $\ell\to\infty$ for some fixed integer $a\ge 2$ and $\ell\in\N$. For convenience, we will assume that $a$ is odd.

\begin{theorem}\label{conv-disc-meas} 
Assume $|\gamma|<\gamma_*$, and let $a$ be an odd integer $\ge2$. Then 
in $L^1(\mathbf P)$,   
$$\mu_{a^\ell}\to\mu\qquad\text{as }\ell\to\infty.$$ 
\end{theorem}

\begin{proof} 
 Given $a$ as above, let us call a function $f$ on $\mathbb{T}^n$  \emph{piecewise constant} if it is constant on all cubes $v+Q_{L}$, $v\in \mathbb{T}^n_{L}$, for some $L=a^{\ell'}$. 
For such $f$ and all $\ell\ge \ell'$,  
\begin{equation}\int f\, d\mu_{a^{\ell}}=\int f\, d\mu_{+,a^{\ell}}.\end{equation}
Indeed, the field $h_{+,a^{\ell}}$ is constant all cubes $v+Q_{a^\ell}$, $v\in \mathbb{T}^n_{a^\ell}$, and Lemma \ref{lem-enhanced-disc} yields that the fields $h_{a^{\ell}}$ and $h_{+,a^{\ell}}$ coincide (in distribution) on the discrete torus $\mathbb{T}^n_{a^\ell}$. 
Thus also the associated LQG measures of all cubes $v+Q_{a^\ell}$, $v\in \mathbb{T}^n_{a^\ell}$,  coincide.

Hence, for piecewise constant functions $f$, the convergence
$$\int f\, d\mu_{a^\ell}\to \int f\, d\mu\qquad\text{as }\ell\to\infty$$
follows from the previous Proposition \ref{thm-conv-enh}.

 For continuous $f$, the claim follows by approximation of $f$ by piecewise constant $f_j$, $j\in \N$.
Indeed,
$$\mathbf E\Big[\Big|\int f\, d\mu_{a^\ell}-\int f_j\, d\mu_{a^\ell}\Big|\Big]\le \mathbf E\Big[\int \Big|f- f_j\Big|\, d\mu_{a^\ell}\Big]=
\int \Big|f- f_j\Big|\, dx\ \to \ 0
$$
as $j\to\infty$, uniformly in $\ell\in\N$, and similarly with $\mu$ in the place of $\mu_{a^\ell}$.
\end{proof}

\subsection{Reduced LQG Measures on the Discrete Tori and their Convergence}

Recall that if $h_L$ is a polyharmonic field on the discrete torus then
\[
h_L^{{-}}\coloneqq {\sf r}_L^{{-}}(h_L)
\]
defines a reduced polyharmonic field on the discrete torus. 
If  $h_L$ is given as 
\[
h_L(v)=\frac1{\sqrt{a_n}}\sum_{z\in\Z^n_L}\frac{1}{\lambda_{L,z}^{n/4}}\,\xi_z\,\varphi_z(v)
\]
 then
\[
h_L^{{-}}(v)=\frac1{\sqrt{a_n}}\sum_{z\in\Z^n_L}\frac{\vartheta_{L,z}}{\lambda_z^{n/4}}\,\xi_z\,\varphi_z(v).
\]

For $\gamma\in\R$,
define the random measure $\mu^{{-}}_{L}$ on $\mathbb{T}_L^n$, called \emph{reduced discrete LQG measure}, by
\[
d\mu^{{-}}_{L}(v)=\exp\Big(\gamma h^{{-}}_{L}(v)-\frac{\gamma^2}2 k^{{-}}_{L}(v,v)\Big)\,dm_L(v) 
\]

\begin{theorem}\label{conv-red-disc} 
Assume $|\gamma|<\sqrt n$  
and let $a$ be an odd integer $\ge2$. Then 
 in $L^1(\mathbf P)$,   
$$\mu^{{-}}_{a^\ell}\to\mu\qquad\text{as }\ell\to\infty.$$ 

\end{theorem}

\begin{proof} 
Let $a$ be given as above and let  $f$ be a function on $\mathbb{T}^n$ which is constant on all cubes $v+Q_{L}$, $v\in \mathbb{T}^n_{L}$, for some $L=a^{\ell'}$. 
Then according to Lemma \ref{lem-natural-dis} for all $\ell\ge \ell'$,  
\begin{equation}\int f\, d\mu^{{-}}_{a^{\ell}}=\int f\, d\mu_{\circ,a^{\ell}}.\end{equation}
Hence, for piecewise constant functions $f$, the convergence
$$\int f\, d\mu^{{-}}_{a^\ell}\to \int f\, d\mu\qquad\text{as }\ell\to\infty$$
follows from the previous Proposition \ref{thm-conv-nat}.
 For continuous $f$, the claim follows by approximation of $f$ by piecewise constant $f_j$, $j\in \N$.
\end{proof}

%
\subsection{Convergence Results for Semi-discrete LQG Measures}
So far, we have studied the LQG measure and the reduced LQG measure on the discrete torus $\mathbb{T}^n_L$ and their convergence properties as $L\to\infty$.
In terms of the polyharmonic field $h_L$ on the discrete torus, we can also define the so-called semi-discrete LQG measure as well as the spectrally reduced semi-discrete LQG measure on the continuous torus. These are the LQG measures associated with the Fourier extension of the discrete field $h_L$ and of the spectrally reduced discrete field $h_L^{{-\circ}}$. All these random measures are  functions of the discrete field $h_L$.

\subsubsection{Semi-discrete LQG Measure}
For $\gamma\in\R$,
define the random measure $\mu_{L,\sharp}$ on $\mathbb{T}^n$, called \emph{semi-discrete LQG measure}, by
\[
d\mu_{L,\sharp}(x)=\exp\Big(\gamma h_{L,\sharp}(x)-\frac{\gamma^2}2 k_{L,\sharp}(x,x)\Big)\,d\Leb^n(x) 
\]
where $h_{L,\sharp}$ denotes the \emph{Fourier extension} of the discrete polyharmonic field $h_L$
and $k_{L,\sharp}$ the associated covariance function   as introduced in~\S\ref{sss:FourierExtension}.

\begin{theorem}\label{conv-four-meas} 
Assume $|\gamma|<\gamma_*\coloneqq\sqrt\frac{n}{e}$. Then 
in $L^1(\mathbf P)$,   $$\mu_{L,\sharp}\to\mu\qquad\text{as }L\to\infty.$$ 
\end{theorem}

\begin{proof} According to Remark \ref{fourier-ext/rest}, the Fourier extension $h_{L,\sharp}$ of the discrete random field $h_L$ coincides in distribution with the field obtained from the continuous field $h$ by regularization with the kernel  $r_{L}^{{+\circ}}$.

To show the convergence of the LQG measures~$\mu_{L,\sharp}$ associated with the Fourier extensions of the polyharmonic field on the torus $\mathbb{T}^n$  to the LQG measure~$\mu$ on~$\mathbb{T}^n$,
 we again verify the necessary assumptions in~\cite[Lem.~4.5]{DHKS}.
 Criteria (ii) and (iii) of Lemma 4.5 in \cite{DHKS}, can be verified exactly as in the proof of Lemma~\ref{lemma-enhanced}.
The uniform integrability of the approximating sequence of random measures follows from Theorem \ref{p:Estimate} in the next section.
\end{proof}

\begin{figure}[htb!]
     \centering
     \begin{subfigure}[b]{0.48\textwidth}
         \centering
         \includegraphics[width=\textwidth]{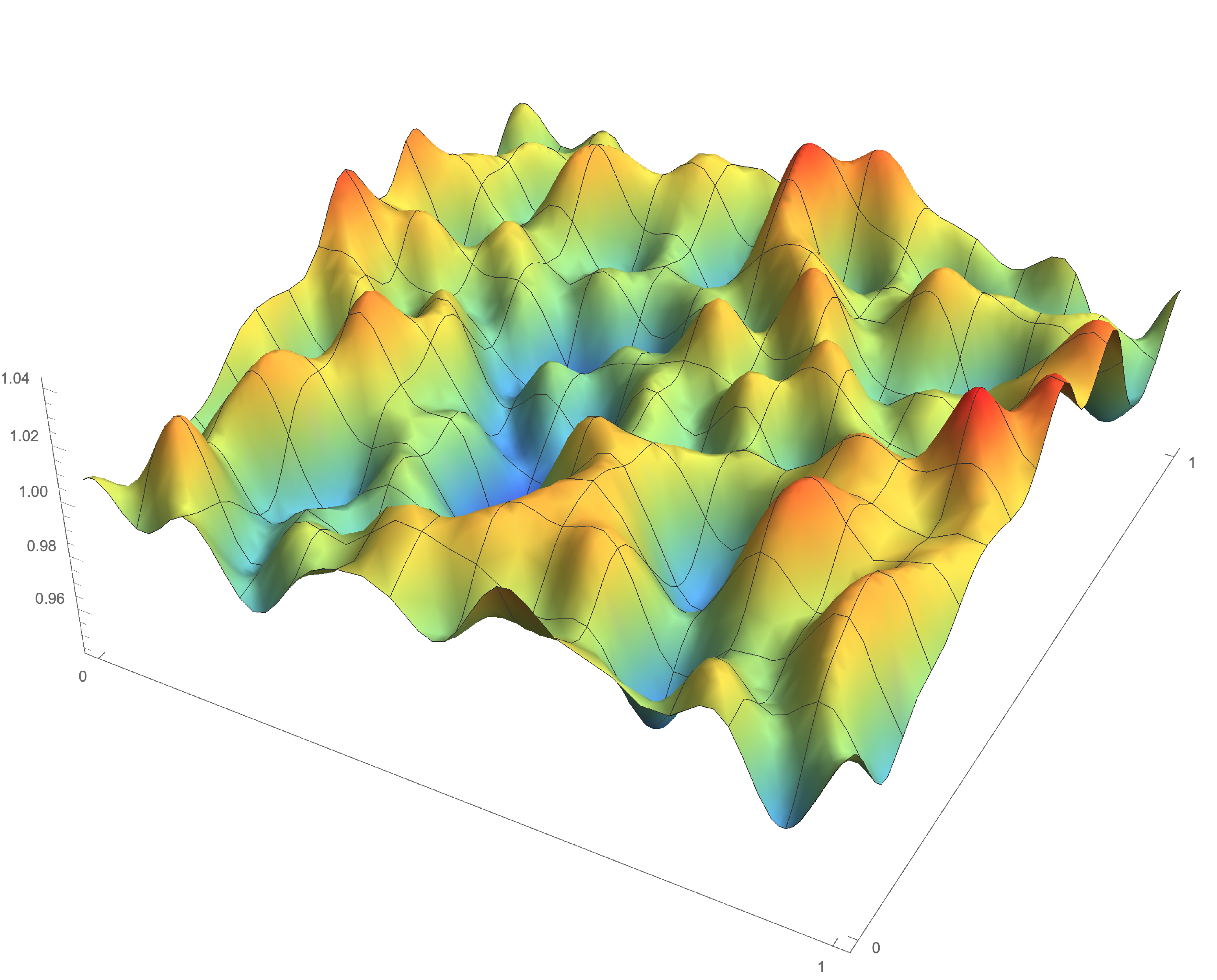}
         \caption{$\gamma=0.01$}
     \end{subfigure}
     \hfill
     \begin{subfigure}[b]{0.48\textwidth}
         \centering
        \includegraphics[width=\textwidth]{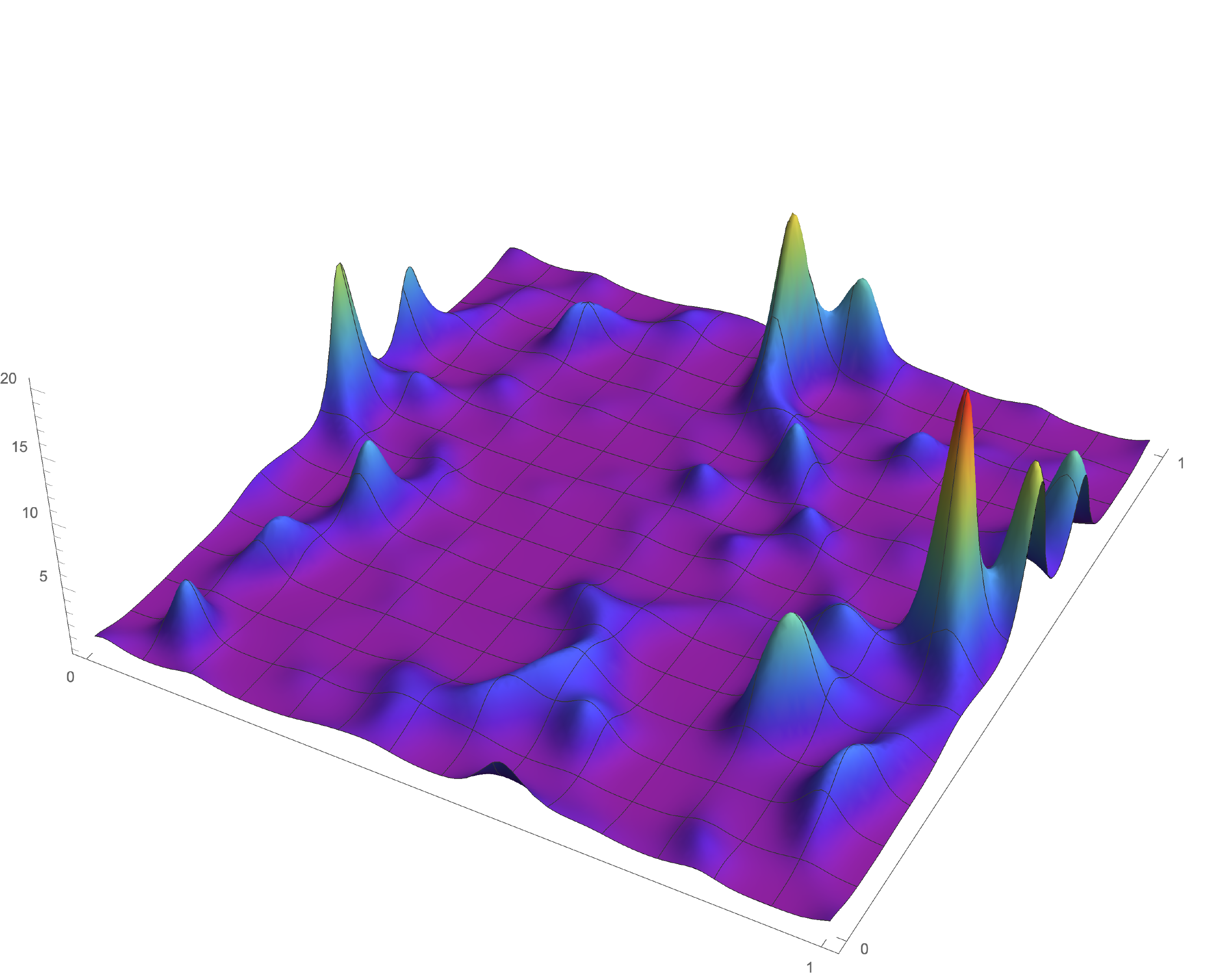}
         \caption{$\gamma=1$}
     \end{subfigure}
        \caption{\footnotesize The Gaussian Multiplicative Chaos~$\mu_{L,\sharp}$ on~$\mathbb{T}^2$ for~$L=15$, different values of~$\gamma$ and same realization of the randomness.}
\end{figure}

\subsubsection{Spectrally Reduced Semi-discrete LQG Measure} 

For $\gamma\in\R$,
define the random measure $\mu^{{-\circ}}_{L,\sharp}$ on $\mathbb{T}^n$, called \emph{spectrally reduced semi-discrete LQG measure}, by
\[
d\mu^{{-\circ}}_{L,\sharp}(x)=\exp\Big(\gamma h^{{-\circ}}_{L,\sharp}(x)-\frac{\gamma^2}2 k^{{-\circ}}_{L,\sharp}(x,x)\Big)\,d\Leb^n(x) 
\]
where $h^{{-\circ}}_{L,\sharp}$ denotes the \emph{Fourier extension} of the spectrally reduced discrete polyharmonic field~$h^{{-\circ}}_L$
and~$k^{{-\circ}}_{L,\sharp}$ the associated covariance function  as introduced in \S\ref{sss:ReducedDiscrete}.
As a corollary to Lemma \ref{fourier-ext/rest} and Proposition \ref{four-meas}, we directly obtain
\begin{theorem}\label{conv-mod-meas} 
Assume $|\gamma|<\sqrt{2n}$. Then in $\mathbf P$-probability and in $L^1(\mathbf P)$,   $$\mu^{{}{-\circ}}_{L,\sharp}\to\mu\qquad\text{as }L\to\infty.$$ 
\end{theorem}

\subsection{Uniform Integrability of Discrete and Semi-discrete LQG Measures}

Finally, we address the question of uniform integrability of approximating sequences of LQG measures.
We provide a self-contained  argument for $L^2$-boundedness, independent of Kahane's work \cite{Kah85}.

\begin{theorem}\label{p:Estimate}
Assume $|\gamma|<\gamma_*\coloneqq\sqrt{\frac{n}e}$. Then
\begin{itemize}
\item[(i)] 
\begin{equation}\label{eq: cont}
 \sup_L\int_{\mathbb{T}^n}\exp\Big(\gamma^2 k_{L,\sharp}(0,y)\Big)\,d\Leb^n(y)<\infty\comma\end{equation}

\item[(ii)] 
\begin{equation}\label{eq: disc}  
\sup_L\int_{\mathbb{T}^n}\exp\Big(\gamma^2 k_{L,\flat}(0,y)\Big)\,d\Leb^n(y)<\infty\fstop\end{equation}

\item[(iii)] 
\begin{equation}\label{eq: nat}  
\sup_L\int_{\mathbb{T}^n}\exp\Big(\gamma^2 k_{+,L}(0,y)\Big)\,d\Leb^n(y)<\infty\fstop\end{equation}

\end{itemize}
\end{theorem}

\begin{proof}
(i) Recall from \eqref{kL-sharp} that for $x,y\in\mathbb{T}^n$, 
\begin{align*}
k_{L,\sharp}(x,y)=&\frac1{a_n}\sum_{z\in\Z^n_L\setminus\{0\}} \frac1{\lambda_{L,z}^{n/2}}\cdot 
\exp\Big(2\pi i\, z\cdot (x-y)\Big)
\end{align*}
Here as before, $\lambda_{L,z}=4L^2\,\sum_{k=1}^n \sin^2\big(\pi z_k/L\big)$.
Given $\epsilon>0$, choose $R>2$ such that $|z|+\frac12\sqrt n\le (1+\epsilon)|z|$ for all $z\in\Z^n$ with $\|z\|_\infty\ge R/2$,
and $S>1$ such that $t\le (1+\epsilon) \sin(t)$ for all $t\in [0, \frac\pi{2S}]$.
Decompose $k_{L,\sharp}$ for $L>R\,S$ into $k_{L,R,S}+ g_{L,R}+f_{L,S}$ with
\begin{align*}
f_{L,S}(x,y)=&\frac1{a_n}\sum_{z\in\Z^n_L\setminus\Z^n_{L/S}} \frac1{\lambda_{L,z}^{n/2}}\cdot 
e^{2\pi i\, z\cdot (x-y)}\comma\qquad
g_{L,R}(x,y)=&\frac1{a_n}\sum_{z\in\Z^n_R\setminus\{0\}} \frac1{\lambda_{L,z}^{n/2}}\cdot 
e^{2\pi i\, z\cdot (x-y)}
\end{align*}
and
\begin{align*}
k_{L,R,S}(x,y)=&\frac1{a_n}\sum_{z\in\Z^n_{L/S}\setminus\Z^n_R} \frac1{\lambda_{L,z}^{n/2}}\cdot 
e^{2\pi i\, z\cdot (x-y)}\fstop
\end{align*}

For fixed $R$, obviously $g_{L,R}(x,y)$ is uniformly bounded in $L,x,y$. 
Similarly,  since $\sin(t)\geq {\frac2\pi}\,t$ for $t\in[0,\pi/2]$ we have that for fixed $S$, 
\begin{align*}\big|f_{L,S}\big|(x,y)&\le\frac1{a_n}\sum_{z\in\Z^n_L\setminus\Z^n_{L/S}} \frac1{\lambda_{L,z}^{n/2}}
\le \frac1{a_n}\sum_{z\in\Z^n_L\setminus\Z^n_{L/S}}\frac1{4^n|z|^n}\\
&\le \frac1{4^n\,a_n}\int_{B_{{\sqrt n}L/2}(0)\setminus B_{L/(2S)}(0)} \frac1{|x|^n}d\Leb^n(x)\\
&=C\,\Big[\log({\sqrt n}L/2)-\log(L/(2S))\Big]=C'<\infty\fstop
\end{align*}

For \eqref{eq: cont} to hold, it thus suffices to prove that
\begin{equation}\label{eq: cont-alt}
\sup_L\int_{\mathbb{T}^n}\exp\Big(\gamma^2 k_{L,R,S}(0,y)\Big)\,d\Leb^n(y)<\infty\comma\end{equation}
for some $R>2$ as above.

 For proving the latter, we follow the argument of the proof of Lemma 2, p. 611 in \cite{zygmund1934some}. To start with, we use the multi-dimensional Hausdorff--Young inequality, which can be found in \cite[p. 248]{folland1999real}:
\begin{align*}
\text{For } p\geq 2 \quad &\int_{\mathbb{T}^n}| k_{L,R,S}(0,y)|^p\, dy\leq \left(\sum_{z\in\Z^n_{L/S}\setminus\Z^n_R}|c(z)|^{p'}\right)^{p-1},
\end{align*}
where $c(z)=\frac1{a_n}\big(4L^2\,\sum_{k=1}^n \sin^2(\pi z_k/L)\big)^{-n/2}$ and $p'\in[1,2]$ is the H\"older-conjugate. Since 
$\pi|z_k|/L\le (1+\epsilon) |\sin(\pi z_k/L)|$ for all $z_k/L$ under consideration,
we have that 
\begin{align*}
\int_{\mathbb{T}^n}|  k_{L,R,S}(0,y)|^p\, dy\leq (1+\epsilon)^{np'(p-1)}\,\left(\frac1{a_n^{p'}}\sum_{z\in\Z^n_{L/S}\setminus\Z^n_R}\frac1{(2\pi|z|)^{np'}}\right)^{p-1}.
\end{align*}
Let $Q_1(z)$ be the unit cube $\prod_{i=1}^n[z-\frac12 e_i,z+\frac12 e_i]$ around $z\in\Z^n$.
Since by assumption
 \begin{align*}
 |x|\leq |z|+\frac12\sqrt{n}\leq (1+\frac12\sqrt{n})|z|\leq  (1+\epsilon)|z|
 \end{align*}
  for all $x\in Q_1(z)$ and all $z$ with $\|z\|_\infty\geq R/2$, we estimate
\begin{align*}
\sum_{z\in\Z^n_{L/S}\setminus\Z^n_R}\int_{Q_1(z)}\frac1{|z|^{np'}}\, dx
\leq &\sum_{z\in\Z^n\setminus\Z^n_R}\left(1+\epsilon\right)^{np'}\int_{Q_1(z)}\frac1{|x|^{np'}}\, dx\\
\le&\left(1+\epsilon\right)^{np'}\int_{\R^n\setminus B_{R/2}(0)}\frac1{|x|^{np'}}\, dx.
\end{align*}
With Cavalieri's principle the integral can be estimated by
\begin{align*}
\int_{\R^n\setminus B_{R/2}(0)}\frac1{|x|^{np'}}\, dx&=\int_{R/2}^\infty \frac{2\pi^{n/2}}{\Gamma(n/2)}r^{n-1}\frac1{r^{np'}}\, dr\\
&=  \frac{2\pi^{n/2}}{\Gamma(n/2)}\frac1{n(p'-1)}\left(\frac R2\right)^{n(1-p')}\leq  \frac{2\pi^{n/2}}{\Gamma(n/2)}\frac p{n}
\end{align*}
since $p'>1$ and $R\ge2$. Hence we obtain for $ k_{L,R,S}$:
\begin{align}\label{zwischen}
\int_{\mathbb{T}^n}|k_{L,R,S}(0,y)|^p\, dy\leq \left(\frac{(1+\epsilon)^2}{2\pi}\right)^{np}\frac1{a_n^p}\left( \frac{2\pi^{n/2}}{\Gamma({n}/2)}\frac p{n}\right)^{p-1}.
\end{align}

Summing these terms over all $p\in\N\setminus\{1\}$  yields
\begin{align*}
\sum_{p\geq 2}\frac{\gamma^{2p}}{p!}\int_{\mathbb{T}^n}|   k_{L,R,S}(0,y)|^p\, dy
\leq &\sum_{p\geq 2}\frac{\gamma^{2p}}{p!}\left(\frac{(1+\epsilon)^{2n}}{(2\pi)^{n}a_n}\right)^p\left(\frac{2\pi^{n/2}}{\Gamma(n/2)}\frac p{n}\right)^{p-1}\\
=&\frac{n\Gamma(n/2)}{2\pi^{n/2}}\sum_{p\geq 2}\frac{\gamma^{2p}}{p!p}\left(\frac{(1+\epsilon)^{2n}}{(2\pi)^{n}a_n}\frac{2\pi^{n/2}p}{n\Gamma(n/2)}\right)^p\\
\sim&\frac{n\Gamma(n/2)}{2\pi^{n/2}}\sum_{p\geq 2}\frac{1}{p\sqrt{2\pi p}}\left(\frac{(1+\epsilon)^{2n}2\pi^{n/2} e \gamma^2}{(2\pi)^n a_n n\Gamma(n/2)}\right)^p,
\end{align*}
where we used Stirling's formula $p!\sim \sqrt{2\pi p}(\frac pe)^p$. The last sum is finite if 
\begin{align}\label{eq:p:Estimate:1}
(1+\epsilon)^{2n}\gamma^2<\frac{(2\pi)^n a_nn\Gamma(n/2)}{2\pi^{n/2}e}
=\frac{n}{e}=\gamma_*^2 \comma
\end{align}
where we inserted $a_n=\frac2{(4\pi)^{n/2}\Gamma(n/2)}$.
Since by assumption $|\gamma|<\gamma_*$ and since $\epsilon>0$ was arbitrary, by appropriate choice of the latter, \eqref{eq:p:Estimate:1} is satisfied.

To treat the cases $p=0$ and $p=1$, observe that $(\int_{\mathbb{T}^n}|   k_{L,R,S}(0,y)|\, dy)^2\leq \int_{\mathbb{T}^n}|   k_{L,R,S}(0,y)|^2\, dy$. Thus,  there exists a constant $C^R_{n,\gamma}$ such that
\begin{align*}
\sum_{p\geq 0}\frac{\gamma^{2p}}{p!}\int_{\mathbb{T}^n}|   k_{L,R,S}(0,y)|^p\, dy\leq C^R_{n,\gamma},
\end{align*}
uniformly in $L$,
and thus in turn
there exists a constant $C^\sharp_{n,\gamma}$ such that
\begin{align*}
\sup_L\,\sum_{p\geq 0}\frac{\gamma^{2p}}{p!}\int_{\mathbb{T}^n}|   k_{L,\sharp}(0,y)|^p\, dy\leq C^\sharp_{n,\gamma},
\end{align*}

which proves \eqref{eq: cont}. 

\medskip

(ii)
To show \eqref{eq: disc},  note that 
$$\int_{\mathbb{T}^n}\exp\Big(\gamma^2 k_{L,\flat}(0,y)\Big)\,d\Leb^n(y)=\int_{\mathbb{T}^n_L}\exp\Big(\gamma^2 k_{L}(0,v)\Big)\,dm_L(v)
$$ for every $L$.
Furthermore, for $p\geq 2$,
\begin{align*}
\int_{\mathbb{T}^n_L}|k_L(0,v)|^p\, dm_L(v)\leq\left(\frac1{a_n^{p'}}\sum_{z\in\Z^n_L\setminus\{0\}}|c(z)|^{p'}\right)^{p-1}.
\end{align*}
Indeed, for $p=2$ this is due to Parseval's identity, and for $p=\infty$ this holds since $|\exp(2\pi iz(x-y))|=1$. The estimate holds for all intermediate $p\in(2,\infty)$ by virtue of the Riesz--Thorin theorem. Then, the proof of \eqref{eq: disc} follows the lines above.

\medskip

(iii)
Recall that $k_{+,L}=q_L\circ k_L^+$ with $k_L$ given is \eqref{kL+}. Thus by Jensen's inequality, \eqref{eq: nat} will follow from
\begin{equation}\label{eq: nat+}  \sup_L\int_{\mathbb{T}^n}\exp\Big(\gamma^2 k^+_{L}(0,y)\Big)\,d\Leb^n(y)<\infty\fstop\end{equation}
To prove this, we argue as before in (i), now with
\begin{align*}
k_{L}^+(x,y)=&\frac1{a_n}\sum_{z\in\Z^n_L\setminus\{0\}} \frac1{\vartheta_{L,z}^2\,\lambda_{L,z}^{n/2}}\cdot 
\exp\Big(2\pi i\, z\cdot (x-y)\Big)
\end{align*}
in the place of $k_{L,\sharp}(x,y)$. For given $\epsilon>0$, choose $R>2$ and $S>1$ as before. In particular, 
then $t\le (1+\epsilon) \sin(t)$ for all $t\in [0, \frac\pi{2S}]$ and thus 
$$1\ge \vartheta_{L,z}\ge (1+\epsilon)^{-n}\qquad\forall z\in\Z^n_{L/S}.$$
Thus decomposing $k_{L}^+$ intro three factors as before and then arguing as before will prove the claim.
\end{proof}

\begin{corollary}\label{p:Integrability}
If $|\gamma|<\gamma_*$, then for each $f\in L^2(\mathbb{T}^n)$,
\begin{itemize}
\item[(i)]  the family $\big(\int_{\mathbb{T}^n} f\,d\mu_{L,\sharp}\big)_{L\in\N}$ is $L^2(\mathbf P)$-bounded, 
\item[(ii)]  the family $\big(\int_{\mathbb{T}^n} f\,d\mu_{L,\flat}\big)_{L\in\N}$ is $L^2(\mathbf P)$-bounded.
\item[(iii)]  the family $\big(\int_{\mathbb{T}^n} f\,d\mu_{+,L}\big)_{L\in\N}$ is $L^2(\mathbf P)$-bounded.
\end{itemize}
\end{corollary}

\begin{proof} (i) Given $f\in L^2(\mathbb{T}^n)$ and $\gamma$ as above, consider the Gaussian variables
$$Y_{L,\sharp}\coloneqq\int_{\mathbb{T}^n} f\,d\mu_{L,\sharp}= \int_{\mathbb{T}^n} \exp\Big(\gamma h_{L,\sharp}(x)-\frac{\gamma^2}2k_{L,\sharp}(x,x)\Big)\, f(x)\,d\Leb^n(x).$$
Then
\begin{align*}
\sup_L\mathbb E\Big[\big|Y_{L,\sharp}\big|^2
\Big]&=
\sup_L\int_{\mathbb{T}^n}\int_{\mathbb{T}^n}\exp\Big(\gamma^2 k_{L,\sharp}(x,y)\Big)\,f(x)\,f(y)\,dm_L(y)\,d\Leb^n(x)
\\
&\le
\sup_L\int_{\mathbb{T}^n}\int_{\mathbb{T}^n}\exp\Big(\gamma^2 k_{L,\sharp}(x,y)\Big)\,f^2(x)\,dm_L(y)\,d\Leb^n(x)
\\
&\le \|f\|^2\cdot \sup_L\int_{\mathbb{T}^n}\exp\Big(\gamma^2 k_{L,\sharp}(0,y)\Big)\,dm_L(y)\le  \|f\|^2\cdot C^\sharp_{n,\gamma}<\infty.
\end{align*}

(ii) Similarly, 
\begin{align*}
\sup_L\mathbb E\Big[\big|Y_{L,\flat}\big|^2
\Big]&\le \|f\|^2\cdot \sup_L\int_{\mathbb{T}^n}\exp\Big(\gamma^2 k_{L,\flat}(0,y)\Big)\,dm_L(y)\le  \|f\|^2\cdot C^\flat_{n,\gamma}<\infty
\end{align*}
for $Y_{L,\flat}\coloneqq \int_{\mathbb{T}^n} f\,d\mu_{L,\flat}=\int_{\mathbb{T}^n} \exp\Big(\gamma h_{L,\flat}(x)-\frac{\gamma^2}2k_{L,\flat}(x,x)\Big)\, f(x)\,d\Leb^n(x)$.

(iii) Analogously.
\end{proof}

{\footnotesize

\begin{thebibliography}{10}

\bibitem{Biskup}
M.~Biskup.
\newblock {Recent progress on the Random Conductance Model}.
\newblock {\em Probab.\ Surv.}, 8:294--373, 2011.

\bibitem{CiprianiDanHazra}
A.~Cipriani, B.~Dan, and R.~S. Hazra.
\newblock {The scaling limit of the membrane model}.
\newblock {\em Ann.\ Probab.}, 47(6):3963--4001, 2019.

\bibitem{CiprianiGraaffRuszel}
A.~Cipriani, J.~de~Graaff, and W.~M. Ruszel.
\newblock {Scaling Limits in Divisible Sandpiles: A Fourier Multiplier
  Approach}.
\newblock {\em {J.\ Theor.\ Probab.}}, 33(4):2061--2088, 2019.

\bibitem{CiprianiHazraRuszel}
A.~Cipriani, R.~S. Hazra, and W.~M. Ruszel.
\newblock {Scaling limit of the odometer in divisible sandpiles}.
\newblock {\em {Probab.\ Theory Relat.\ Fields}}, 172(3-4):829--868, dec 2017.

\bibitem{DHKS}
L.~{Dello Schiavo}, R.~Herry, E.~Kopfer, and K.-T. Sturm.
\newblock {Conformally Invariant Random Fields, Quantum Liouville Measures, and
  Random Paneitz Operators on Riemannian Manifolds of Even Dimension}.
\newblock {\em {arXiv:2105.13925}}, 2021.

\bibitem{DKS20}
{Dello Schiavo, L.}, {Kopfer, E.}, and {Sturm, K.-T.}
\newblock {A Discovery Tour in Random Riemannian Geometry}.
\newblock {\em {arXiv:2012.06796}}, 2020.

\bibitem{DuplantierRhodesSheffieldVargas}
B.~Duplantier, R.~Rhodes, S.~Sheffield, and V.~Vargas.
\newblock {Log-correlated Gaussian Fields: An Overview}.
\newblock In {\em Progress in Mathematics}, pages 191--216. Springer
  International Publishing, 2017.

\bibitem{folland1999real}
G.~B. Folland.
\newblock {\em Real analysis: modern techniques and their applications},
  volume~40.
\newblock John Wiley \& Sons, 1999.

\bibitem{Kah85}
{Kahane, J.-P.}
\newblock {Sur le Chaos Multiplicatif}.
\newblock {\em {Ann.\ sc.\ math.\ Qu{\'{e}}bec}}, 9(2):105--150, 1985.

\bibitem{Kurt}
N.~Kurt.
\newblock {Maximum and entropic repulsion for a Gaussian membrane model in the
  critical dimension}.
\newblock {\em Ann.\ Probab.}, 37(2):687--725, 2009.

\bibitem{LevineMuruganPeresUgurcan}
L.~Levine, M.~Murugan, Y.~Peres, and B.~E. Ugurcan.
\newblock {The Divisible Sandpile at Critical Density}.
\newblock {\em Ann.\ Henry Poincar{\'{e}}}, 17:1677--1711, 2016.

\bibitem{LodSheSunWat16}
A.~Lodhia, S.~Sheffield, X.~Sun, and S.~S. Watson.
\newblock {Fractional Gaussian fields: A survey}.
\newblock {\em {Probab.\ Surveys}}, 13:1--56, 2016.

\bibitem{Sakagawa}
H.~Sakagawa.
\newblock {Entropic repulsion for a Gaussian lattice field with certain finite
  range interaction}.
\newblock {\em J.~Math.\ Phys.}, 44(2939), 2003.

\bibitem{Schweiger}
F.~Schweiger.
\newblock {The maximum of the four-dimensional membrane model}.
\newblock {\em Ann.\ Probab.}, 48(2):714--741, 2020.

\bibitem{Sha16}
A.~Shamov.
\newblock {On Gaussian multiplicative chaos}.
\newblock {\em {J.\ Funct.\ Anal.}}, 270:3224--3261, 2016.

\bibitem{Sheffield}
S.~Sheffield.
\newblock {Gaussian free fields for mathematicians}.
\newblock {\em Probab.\ Theory Relat.\ Fields}, 139:521--541, 2007.

\bibitem{zygmund1934some}
A.~Zygmund.
\newblock Some points in the theory of trigonometric and power series.
\newblock {\em Trans.\ Amer.\ Math.\ Soc.}, 36(3):586--617, 1934.

\end{thebibliography}
%

}

\end{document}